\pdfoutput=1

\documentclass[11pt]{article}
\usepackage[utf8]{inputenc}  
\usepackage[british]{babel}
\usepackage{amsmath,amssymb,amsfonts,amsthm,bbm}
\usepackage{tikz}
\usepackage{enumerate,enumitem,mdframed}

\usepackage{crossreftools,booktabs}

\makeatletter
\newcommand{\optionaldesc}[2]{%
  \phantomsection
  #1\protected@edef\@currentlabel{#1}\label{#2}%
}
\makeatother

\usepackage{microtype}

\usepackage[mono=false]{libertine}
\usepackage[T1]{fontenc}
\usepackage[cmintegrals,libertine]{newtxmath}
\usepackage[cal=euler, calscaled=.95, frak=euler]{mathalfa}
\useosf
\usepackage[a4paper,vmargin={2cm,2cm},hmargin={2.25cm,2.25cm}]{geometry}
\usepackage[font=sf, labelfont={sf,bf}, margin=1cm]{caption}
\usepackage{numprint,ulem}
\newlist{compitem}{itemize}{4}
\setlist[compitem,1]{nolistsep,label=$\bullet$}

\usepackage{makecell} 

\SetLabelAlign{CenterWithParen}{\makebox[1.6cm]{#1}}

\usepackage{hyperref}
\hypersetup{
	colorlinks=true,
		citecolor=blue!60!black,
		linkcolor=red!60!black,
		urlcolor=green!40!black,
		filecolor=yellow!50!black,
	breaklinks=true,
	pdfpagemode=UseNone,
	bookmarksopen=false,
}
\usepackage[capitalise]{cleveref}

\colorlet{link}{red!60!black}

\usepackage[section]{placeins} 

\theoremstyle{plain}
\newtheorem{theorem}{Theorem}[section]
\newtheorem{lemma}[theorem]{Lemma}
\newtheorem{proposition}[theorem]{Proposition}
\newtheorem{corollary}[theorem]{Corollary}

\theoremstyle{definition}

\theoremstyle{remark}
\newtheorem{remark}[theorem]{Remark}

\makeatletter
\newcommand{\labeltext}[3][]{%
    \@bsphack%
    \csname phantomsection\endcsname
    \def\tst{#1}%
    \def\labelmarkup{\textcolor{link}}
    \def\refmarkup{}%
    \ifx\tst\empty\def\@currentlabel{\refmarkup{#2}}{\label{#3}}%
    \else\def\@currentlabel{\refmarkup{#1}}{\label{#3}}\fi%
    \@esphack%
    \labelmarkup{#2}
}
\makeatother
\newcommand{\N}{\mathbb{N}} 
\newcommand{\R}{\mathbb{R}} 
\renewcommand{\P}{{\mathbb{P}}} 
\newcommand{\E}{\mathbb{E}} 
\newcommand{\1}{\mathbbm{1}} 

\newcommand{\proba}[1]{\mathbb{P}\left(#1\right)} 
\newcommand{\probak}[1]{\mathbb{P}^k\left(#1\right)} 

\newcommand{\e}{\varepsilon} 

\newcommand{\T}{\mathcal{T}} 

\newcommand{\h}{\mathsf{h}} 

\newcommand{\cv}[1][n]{\enskip\mathop{\longrightarrow}^{}_{#1 \to \infty}\enskip}
\newcommand{\cvloi}[1][n]{\enskip\mathop{\longrightarrow}^{(d)}_{#1 \to \infty}\enskip}
\newcommand{\cvps}[1][n]{\enskip\mathop{\longrightarrow}^{a.s.}_{#1 \to \infty}\enskip}
\newcommand{\cvproba}[1][n]{\enskip\mathop{\longrightarrow}^{\P}_{#1 \to \infty}\enskip}

\DeclareMathOperator*{\limit}{\longrightarrow}
\newcommand{\K}{\mathbb{K}}
\newcommand{\GH}{\text{GH}}
\DeclareMathOperator*{\supp}{supp}

\newcommand{\GP}{\text{GP}}

\newcommand{\GHP}{\text{GHP}}

\usepackage{mathtools}

\let\originalleft\left
\let\originalright\right
\renewcommand{\left}{\mathopen{}\mathclose\bgroup\originalleft}
\renewcommand{\right}{\aftergroup\egroup\originalright}

\newcommand{\fath}{\mathfrak f}
\newcommand{\trail}{\mathcal S}

\newcommand{\Part}{\mathcal{P}}
\title{Scaling limit of trees with vertices of fixed degrees and heights}
\author{
Arthur Blanc-Renaudie 
\thanks{Universit\'e Paris Saclay, \textsf{ablancrenaudiepro@gmail.com}. Supported by ERC consolidator grant 101087572 (SuPerGRandMa). 
} 
\qquad  
Emmanuel Kammerer 
\thanks{CMAP, \'Ecole polytechnique, Institut Polytechnique de Paris, 91120 Palaiseau, France, \textsf{emmanuel.kammerer@polytechnique.edu}
} 
}

\begin{document}
\date{\today}
\vspace{-2cm}
\maketitle 
\begin{abstract} We consider large uniform random trees where we fix for each vertex its degree and height. We prove, under natural conditions of convergence for the profile, that those trees properly renormalized converge. To this end, we study the paths from random vertices to the root using coalescent processes. As an application, we obtain scaling limits of Bienaymé--Galton--Watson trees in varying environment.
\end{abstract}

\begin{figure}[!ht]
\label{fig:ssimus_intro}
\centering
\includegraphics[width=0.65\linewidth]{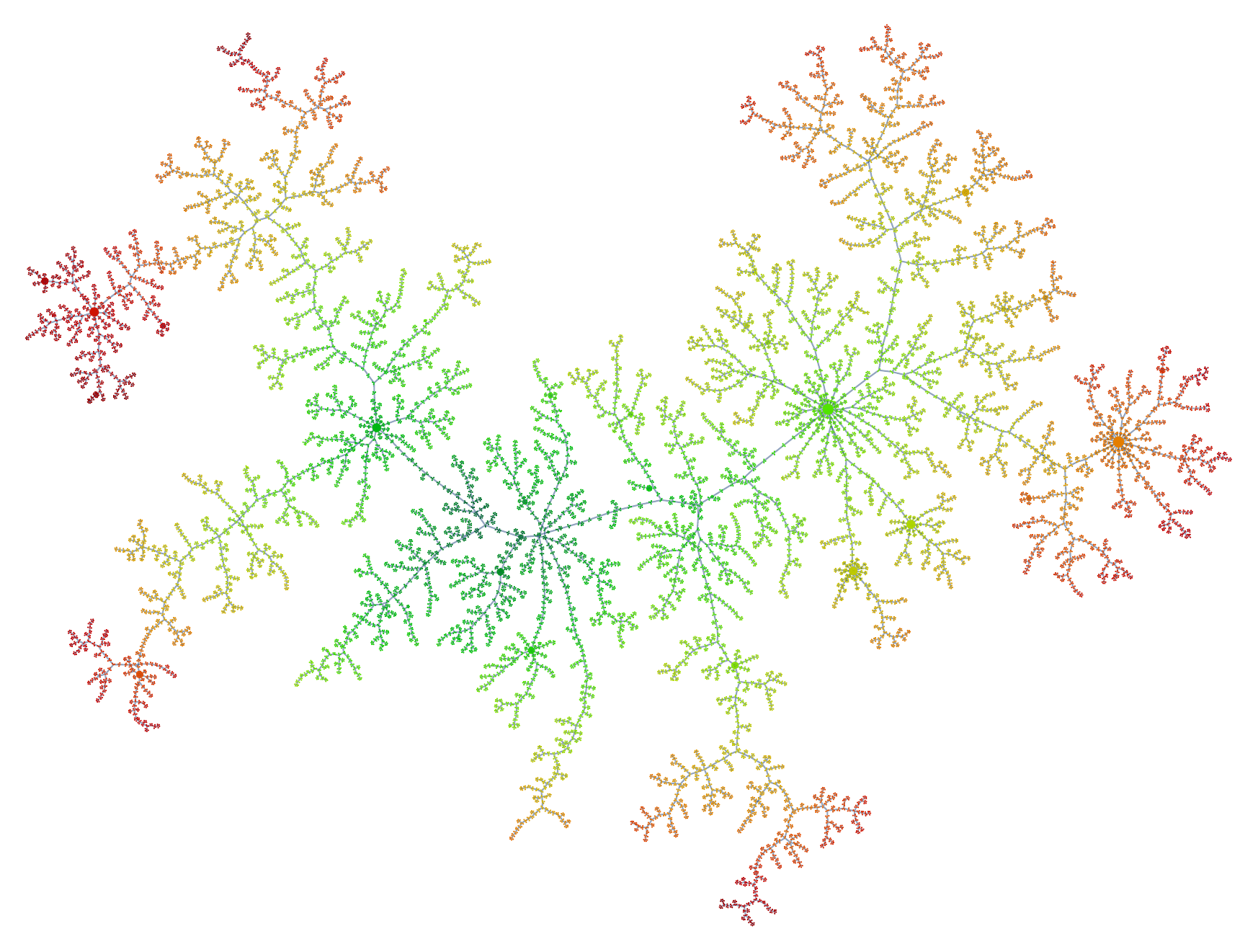}%
\qquad  \includegraphics[width=0.3\linewidth]{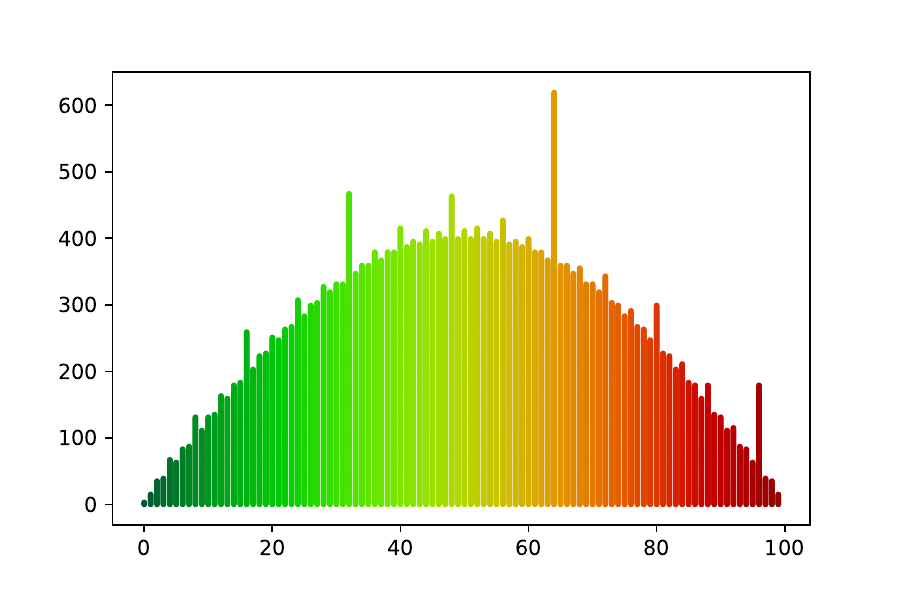}
\caption{Simulation of a uniform random tree of height $100$ whose profile scales towards $t \mapsto \sin( \pi t)$ with vertices of degrees $0$ and $4$ with some additional vertices with large degrees. The profile is drawn on the right. The low vertices are green while the high vertices are red. 
}
\end{figure}
\paragraph{Acknowledgements.} We thank the anonymous referee for their detailed reading and for their useful comments, which helped to improve the paper.
\section{Introduction}
\subsection{Model}
Trees with fixed degree sequence are universal models of random trees and are related to many other objects: Galton--Watson trees, configuration model, bipartite planar maps with fixed face degrees\dots The aim of this paper is to study 
a variant of this model where we fix for each vertex its degree and also its height. More precisely, we are interested in the shortest-path distance on the rooted labelled trees defined below.

Consider a set of vertices $\{(i,j)\}_{i\geq 0,j\geq 1}$. We consider 
$(i,j)$ to have height $i$ and degree (number of children) $d_{i,j}^n\geq 0$. We use a superscript $n\in \N$ as we study the behaviour of the trees as the degrees vary. For convenience, we assume that 
on every line $i\ge 0$ the sequence $(d^n_{i,j})_{j\ge 1}$ is non-increasing. Also, to ensure that every vertex $(i,j)$ of degree $d^n_{i,j}\geq 1$ is connected to the root $(0,1)$, we assume that  $d^n_{0,2}=0$ and
\begin{equation}
\forall i \ge 0, \qquad
    \# \{ j \ge1 ; \ d^n_{i+1,j}>0\} \le D^n_i \coloneqq \sum_{j\ge 1} d^n_{i,j}<\infty. \label{eq:CoherenceDegree}
\end{equation}
We construct the plane tree $\T^n$ height by height:  $(0,1)$ is the root, and independently and uniformly  for every height $i\geq 0$ we attach each vertex in $\{(i+1,j)\}_{1\leq j \leq D_i^n}$ to a vertex in $\{(i,j); \ {j \ge 1} , d^n_{i,j} \ge 1\}$ so that every vertex $(i,j)$ has exactly $d_{i,j}^n$ children (see Figure \ref{fig:Exemple}). Note that $\T^n$ has height $\h^n \coloneqq \inf \{i\ge 0; \ d^n_{i,1}=0 \}$. {The tree $\T^n$ is equipped with the shortest path distance, where each edge has length $1$.}

\begin{figure} 
\centering
\includegraphics[scale=0.44]{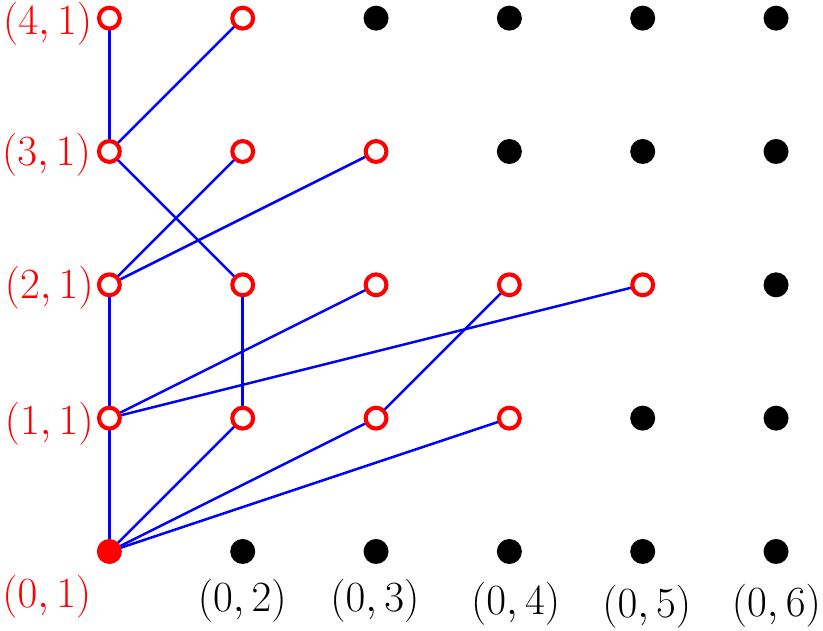}
\caption{An example of tree $\T$ such that $d_{0,1}=4$, $d_{1,1}=3$, $d_{1,2}=1$, $d_{1,3}=1$, $d_{2,1}=2$, $d_{2,2}=1$, $d_{3,1}=2$ and all the other degrees are null. The vertices of $\T$ are represented in red, the root as a red dot, {and} the other vertices as black dots. Here, $D_0=4, D_1=5, D_2=3, D_3=2, D_4=0$.} \label{fig:Exemple}
\end{figure}
\subsection{Main results}
We now present our assumptions on the degrees for the scaling limit of $\T^n$. Without loss of generality, our assumptions are written so that the random metric space $\T^n/n$ obtained after dividing the distances by $n$ converges in distribution. First we assume that the profile converges weakly: {we assume that} for {some} non-atomic probability measure  $\nu$ with compact support $\supp \nu = [0,1]$, we assume, writing $\delta_x$ for the Dirac measure on $x$, 
\begin{mdframed}[linecolor=black!60]
\begin{description}[topsep=0pt,itemsep=-1ex,partopsep=1ex,parsep=1ex,labelwidth=1.6cm,leftmargin=!,align=CenterWithParen]
\item[\textcolor{link}{\optionaldesc{$(\mathsf{Lim}_{\nu})$}{hyp:naissance}}]  $ \sum_{i\geq 0} \delta_{i/n}D^n_{i}/(\sum_{i\ge 0} D^n_i)$ converges weakly towards $\nu$ .
\end{description}
\end{mdframed}
Note that the above assumption is equivalent to {requiring} a scaling limit for the height of random vertices. It also implies that $\liminf \h^n/n \ge 1$

Then, to show the convergence of $(\T^n/n)$, we look at the \textbf{genealogies} of random vertices, \emph{that is their paths to the root}. The rate at which genealogies merge actually only depends on the degrees. We assume {informally that this rate converges}. We distinguish two cases for such merging: either the paths merge at a vertex with large degree, or they merge among the numerous vertices of small degrees. 

Concretely, we use a locally finite measure $\rho$ on $(0,1)$ to describe the merging rate due to small vertices. We also describe the limit of large degrees, using non negative parameters $\Theta=((\theta_j(t))_{j\ge 1})_{t\in (0,1)}$ satisfying {that}: 
\begin{itemize}
\item For every height $t\in (0,1)$, the sequence $(\theta_j(t))_{j\geq 1}$ is non-increasing.
\item For every $t\in (0,1)$, we have
$
\sum_{j\ge 1} \theta_j(t) \le 1.
$
\item The set $\{(j,t):\theta_j(t)\neq 0\}$ is countable. Moreover, for {all} real numbers $0<x<y<1$:
\begin{equation}\label{eq:pas de coalescence avec proba non nulle}
\sum_{x \le t \le y} \sum_{j\ge 1} \theta_j(t)^2 < \infty.
\end{equation}
\end{itemize}

We assume {that} as $n\to \infty$:

\vbox{
\begin{mdframed}[linecolor=black!60]
\begin{description}[topsep=0pt,itemsep=-1ex,partopsep=1ex,parsep=1ex,labelwidth=1.6cm,leftmargin=!,align=CenterWithParen]
\item[\textcolor{link}{\optionaldesc{$(\mathsf{Lim}_{\rho})$}{hyp:coalescence}}] 
{We have the following vague convergence of measures on $((0,1),\mathcal{B}((0,1)))$: 
\[\sum_{1\le i<n,j:d_{i,j}>0} \delta_{i/n} \binom{d^n_{i,j}}{2}/{\binom{D^n_{i}}{2}} \limit_{n\to \infty} \rho + \sum_{j,t:\theta_{j}(t)>0} \theta_j(t)^2\delta_{t}.
\] 
}
\item[\textcolor{link}{\optionaldesc{$(\mathsf{Lim}_{\Theta})$}{hyp:atomes}}] 
We have the following {vague} convergence on $((0,1)\times \R_{+}^*,\mathcal{B}((0,1)\times \R_{+}^*))$
\[
\sum_{1 \le i<n,j: d_{i,j}>0}  \delta_{(i/n, d^n_{i,j}/D^n_i)}
\cv[n]
\sum_{t,j:\theta_{j}(t)>0} \delta_{(t, \theta_j(t))}.
\]
\item[\textcolor{link}{\optionaldesc{$(\mathsf{Lim}_{\neq})$}{hyp:splitatomes}}] For every $\e>0$ there exists $\delta(\e)>0$ such that for every $n>0$, $i, i'$ with $\e n <i,i'<(1-\e ) n${,} $d^n_{i,1}>\e D_i^n$ and $d^n_{i',1}>\e D_{i'}^n$ we either have $i=i'$ or $|i-i'|>\delta(\e) n$. 
\end{description}
\end{mdframed}
}
We require the technical condition \ref{hyp:splitatomes} because the merging rate behaves really differently {depending on whether} vertices of large degree are at the same height or not. Morally, we use \ref{hyp:splitatomes} to ensure that in the limit tree vertices of large degree that are at the same height correspond to vertices at the same height in $\T^n$.

Lastly, while the previous assumptions are sufficient to prove the convergence of the distances between random vertices, we want the following tightness assumption to ensure that the limit space is well defined:

\begin{mdframed}[linecolor=black!60]
\begin{description}[topsep=0pt,itemsep=-1ex,partopsep=1ex,parsep=1ex,labelwidth=1.6cm,leftmargin=!,align=CenterWithParen]
\item[\textcolor{link}{\optionaldesc{$(\mathsf{Tight}_{\mathsf{GP}})$}{hyp:tightGP}}] $\nu$ has support $[0,1]$ and for all $0<a<b<1$, $\rho([a,b])>0$ or $\sum_{j\in \N,t\in 
[a,b]:\theta_{j}(t)>0} \theta_j(t)=\infty$. 
\end{description}
\end{mdframed}
\begin{remark} Aldous, Camarri, and Pitman \cite{introICRT1,introICRT2} used a similar condition to define inhomogeneous continuum random trees. Namely, with their notation $\theta_0>0$ or $\sum_{i\geq 1} \theta_i=\infty$. 
\end{remark}

\begin{theorem} \label{th:GP} Assume that for some $\nu, \rho,\Theta$ the conditions \ref{hyp:naissance},  \ref{hyp:coalescence}, \ref{hyp:atomes}, \ref{hyp:splitatomes}, \ref{hyp:tightGP} hold. Then there exists a random measured metric space $\T(\nu,\rho,\Theta)$ such that the following weak convergence holds for the Gromov--Prokhorov topology (see Appendix \ref{sec:background} for more details)
\[ \T^n/n \cvloi \T(\nu,\rho,\Theta),\]
{where we recall that $\T^n/n$ is the metric space obtained by dividing the shortest path distance on $\T^n$ by $n$, which is also equipped} with the uniform probability measure on its vertices.
\end{theorem}
The Gromov--Prokhorov topology essentially describes the distances between random vertices in $\T^n$. To study geometric variables that depend on all the vertices such as the diameter of $\T^n$, one needs a stronger topology. To this end, we as usual strengthen the convergence to  the Gromov--Hausdorff--Prokhorov topology (see Appendix \ref{sec:background}). We use the following tightness assumptions.

{Let $\|D^n\|_\infty\coloneqq \max_{0\leq i\leq \h^n} D_i^n$.} Then, for every $i\in \N, k>0$, let 
\begin{equation}\label{eq def tau}
\tau^n_{i}(k)\coloneqq\sum_{j\ge 1} \frac{d^n_{i,j}}{D^n_i} \min\left (k\frac{d^n_{i,j}-1}{D^n_i-1},1 \right),\end{equation}
with the convention $\tau^n_i(k)=\infty$ when $D^n_i\in \{0,1\}$. Then, for every $a<b$ let $\tau^n_{a,b}(k)\coloneqq\sum_{i\in [a,b]}\tau^n_i(k)$. {We assume that:}
\begin{mdframed}[linecolor=black!60]
\begin{description}[topsep=0pt,itemsep=-1ex,partopsep=1ex,parsep=1ex,labelwidth=1.6cm,leftmargin=!,align=CenterWithParen]
\item[\textcolor{link}{\optionaldesc{$(\mathsf{Tight}_\mathsf{GHP})$}{hyp:tightGHP}}] For every $0<\alpha<\beta<1$, for every $n\in \N$, $k\leq \|D^n\|_\infty$ large enough {(i.e.\@ there exist $n(\alpha, \beta), k(\alpha, \beta)$ such that for all $n\geq n(\alpha,\beta)$, $k\geq k(\alpha,\beta)$)}, for every $i\in [\alpha n,\beta n]$, we have
$$\tau_{i,i+n/(\log \log k)^2}^n(k)\geq \log(k). $$
\end{description}
\end{mdframed}
\begin{theorem} \label{th:GHP} Under the setting of Theorem \ref{th:GP} if furthermore $\h^n\sim n$ and the condition \ref{hyp:tightGHP} hold then 
$\T^n/n\limit \T(\nu,\rho,\Theta)$ 
weakly for the Gromov--Hausdorff--Prokhorov topology.
\end{theorem}
\begin{remark} 
Writing $\psi_i^n:k\mapsto k\tau_i^n(k)$, one may see $\psi$ as the analog for our model of the Laplace exponent of Lévy trees and Bienaymé--Galton--Watson trees. {Indeed, as usually done in the literature (see e.g.\@ \cite{DLG05}), one may associate the degrees with the jumps of the Lévy process; and in the same way that the jumps contribute linearly or quadratically to the Laplace exponent of Lévy process $\psi(l)$ according to whether they are larger or smaller than $1/l$, the degrees $d_{i,j}^n$ contribute linearly or quadratically to $\tau_i^n(k)$ according to whether they are larger or smaller than $D_i^n/k$.} 
Moreover, \ref{hyp:tightGHP} may be seen as an integrability condition for $1/\psi$ { as in Equation (7) of \cite{DLG05}}. The main difference is that we morally need this integrability condition  at all heights. Due to that, although condition \ref{hyp:tightGHP} is sub-optimal, it seems strong enough to deal with the applications  below. 
\end{remark}
\subsection{Connections with previous works and applications}

When one only prescribes the degrees of the vertices but not their heights, one obtains the well studied model of uniform trees with fixed degree sequence, whose scaling limit is studied for example in \cite{BM14}, \cite{BR21} and also in \cite{Mar18} with applications to Bienaymé--Galton--Watson trees and random planar maps. This model is the analogue of the configuration model for random trees. Note that if one conditions on the heights of the vertices in a uniform tree with fixed degree sequence, then one obtains a uniform random tree with fixed degrees and heights. Moreover, the paper \cite{AUB20} gives conditions for the scaling limit of the profile of uniform trees with fixed degree sequence. Checking our assumptions in order to give a new proof of the scaling limit of uniform trees with fixed degree sequence thus seems feasible. {In this example and in the next two paragraphs, the associated $\nu$, $\rho$ and $\theta$ may be random.}

Besides, if one considers a Galton--Watson process in varying environment (GWVE) and conditions on the evolution of the total population, i.e.\@ the profile of the tree, and on the number of children of each individual, i.e.\@ the degrees, then the corresponding tree is a uniform random tree with fixed degrees and heights. As a result, our result provides the scaling limit of GWVE trees as soon as the hypothesis are satisfied by the profile and the degrees. The first scaling limit of the profile for GWVE trees was obtained by Kurtz \cite{Kur78} under a third moment assumption. These results where subsequently extended and general conditions for the convergence of the profile can be found in \cite{BS15} or in \cite{FLL22}. See \cref{sec:appBGWTVE} for an application of our results to GWVE trees using the scaling limit of the profile of \cite{BS15}. {Informally, under some assumptions which essentially ensure that the scaling limit of the profile from \cite{BS15} holds, we obtain scaling limits of GWVE trees.}

In the same direction, the scaling limit of the size of the $n$-th generation conditioned on survival up to time $n$ was obtained in \cite{Ker20} for GWVE under mild assumptions. Under a second moment assumption, the genealogy of the individuals at this $n$-th generation in GWVE trees was studied in \cite{HPP22} in the critical case, and in \cite{BFRS22} in the near critical case with a GHP scaling limit of the genealogy of the $n$-th generation in this latter work. In \cite{CKKM22}, for i.i.d. environments, in the critical case and under a second moment assumption, a GHP scaling limit towards the Brownian continuous random tree was obtained. Although we do not wish to pursue in that direction, our results should be useful to study the stable case.

Finally, let us briefly mention that our approach is similar to the one we recently developed with Bellin and Kortchemski in \cite{BBKK23+,BBKKbis23+} to study (critical) uniform attachment trees with freezing. However, as we no longer work in a Brownian case only, we need to use new techniques and ideas.

\paragraph{Plan of the paper:} The rest of the paper is organized as follows. In \cref{sec:deflimit} we define the limit tree appearing in Theorems \ref{th:GP} and \ref{th:GHP}. In  \cref{sec:proofGP} we prove \cref{th:GP}. In \cref{sec:proofGHP} we show \cref{th:GHP}. Finally in \cref{sec:appBGWTVE}, which can be read right after the introduction, we discuss an application of our result to Bienaym\'e--Galton--Watson trees in varying environment. Appendix \ref{sec:background} recalls the standard notions about the GP, GH, and GHP topologies. \cref{sec:leaf} recalls the classical notion of leaf-tightness.
\section{Definition of the limit} \label{sec:deflimit}
In this section, we define the limit trees appearing in Theorem \ref{th:GP} in a rather implicit way: we first construct, using a growth-coalescent process, a random infinite matrix which describes the distances between i.i.d.\@ random points. We then prove, under the condition \ref{hyp:tightGP}, leaf-tightness for this matrix, hence defining the limiting random measured metric space (see Proposition \ref{B.1}).

\subsection{Description of the distance matrix between uniform vertices in the limit tree}\label{sec: coalescent continu}
We go downward in the tree from random vertices to the root, and describe how their genealogies merge. This description gives both their heights and the heights of their most recent common ancestors, and so the distances between those vertices. We start with an informal construction of the genealogies' coalescent{.}

Let $k$ particles be born at i.i.d.\@ times (heights) of law $\nu$. On each interval $[t,t-\mathrm{d} t]$, two particles merge into a cluster at rate $\rho(\mathrm{d} t)$. Moreover, at every time $t$ such that $\theta_1(t)>0$, the clusters merge according to the sub-probability measure $(\theta_j(t))_{j \ge 1}$ as follows: independently, each cluster is associated with some random integer $j$ in $\N \cup \{0\}$ with probability $\theta_j(t)$ when $j\ge 1$ and $\theta_0(t)\coloneqq1-\sum_{j'\ge 1} \theta_{j'}(t)$ when $j=0$. 
Then, the clusters  associated to $0$ do not merge whereas for all $j \ge 1$, the clusters associated to $j$ merge together into one cluster. Finally, at time zero, all the {clusters} which have not merged yet coalesce.

Here the birth time of the particle should be seen as the height of random vertices in the tree, and the merging time of two clusters as the height of their most recent common ancestor. And, depending on the degree of those ancestors ({$2$} or $\infty$), the coalescences occur differently (using respectively $\rho$ or $\Theta$).

Formally, let us use the same notation as for coalescent processes where each particle is associated with an integer, where clusters are described by sets of integers and where each coalescence is performed by replacing two clusters by their union. See e.g.\@ \cite{Pit06} or \cite{Ber06} for the use of this notation in the classical theory of coalescent and fragmentation processes.
We consider three types of independent random events:
\begin{itemize}
\item[] \textbf{Birth:} For every $r\in \N$, let $H_r$ be a random variable in $[0,1]$ of law $\nu$; 
\item[] \textbf{Small merge:} For every $q<r\in \N$, let $\Gamma_{q,r}$ be a Poisson point process on $[0,1]$ of intensity $\rho(\mathrm{d} t)$;
\item[] \textbf{Large merge:} For every $t$ such that $\theta_1(t)>0$ and $r\in \N$, let $\Theta_{t,r}$ be a random variable of law given by $\proba{\Theta_{t,r}=j}=\theta_j(t)$ for every $j\in \N\cup \{0\}$.  We write $S^\Theta(t,j)\coloneqq \{r\in \N,\  \Theta_{t,r}=j \}$ for all $j \ge 1$.
\end{itemize}
From there we describe a càglàd (left continuous with right limit) jump process $((\Part_r^k(t))_{1\leq r \leq k})_{1\geq t \geq 0}$ which is called a growth-coalescent process. The process starts at $\Part_1^k(1)=\Part_2^k(1)=\dots =\Part_k^k(1)=\emptyset $ and its jumps occur only at some time $t$ satisfying one of the following conditions: 
\begin{itemize}
\item[] \textbf{Birth:} If $t=H_r$ for some $1\leq r \leq k$, then $\Part_r^k(t^+)=\{\emptyset\}$, and $\Part_r^k(t)=\{r\}$;
\item[] \textbf{Small merge:} If $t\in \Gamma_{q,r}$ for some $q<r$, then whenever $\Part_q^k(t^+)\neq \emptyset$ and $\Part_r^k(t^+)\neq \emptyset$ we have $\Part_q^k(t)=\Part_q^k(t^+)\cup \Part_r^k(t^+)$ and $\Part_r^k(t)=\emptyset$;
\item[] \textbf{Large merge:} If $\theta_1(t)>0$, then for every $j\in \N$ such that $\# (S^\Theta (t,j) \cap [k])\ge 2$, we write $m^k(t,j)$ for the smallest integer $r\in S^\Theta(t,j)\cap[k]$ such that $\Part_r^k(t^+)\neq \emptyset$. We have:
\[ \forall r\in (S^\Theta(t,j)\cap[k]) \setminus \{m^k(t,j) \}, \quad \Part_r^k(t)=\emptyset, \quad \text{and} \quad  \Part_{m^k(t,j)}^k(t)=\bigcup_{r\in S^\Theta(t,j)\cap[k]} \Part_{r}^k(t^+).\]
\end{itemize}
There is a final jump at time $0$ to reach the configuration $\Part_1^k(0)=\{1,\dots, k\}$ and $\Part_2^k(0)=\dots=\Part_k^k(0)=\emptyset$.
\begin{lemma}\label{lemme coalescent continu bien defini} The process $((\Part_r^k(t))_{1\leq r \leq k})_{0\leq t\leq 1}$ is a.s.\@ well defined.
\end{lemma}
\begin{proof}
To see that such a process is well defined, it suffices to focus on the number of possible jumps.
Births only occur at $\{H_r\}_{1\leq r \leq k}$ which is a finite subset of $(0,1)$ since $\nu$ has no atoms at zero or one. Note that for $1\leq r \leq k$ and $t>H_r$ we must have $\Part_r^k(t)=\emptyset$ so that the highest jump occurs at time $\max_{1\leq r \leq k} H_r$.

Then, for every $a>0$, we 
prove that there is only a finite amount of possible times for small merges and large merges in $[a, \max_{1 \le r \le k} H_r ]$, so that the process $((\Part_r^k(t))_{1\leq r \leq k})_{0\leq t\leq 1}$ is well defined on $[a, 1 ]$. Since a.s. $\Part^k_1(0)=\{1,\dots, k\}$ and $\Part_2^k(0)=\dots=\Part_k^k(0)=\emptyset$, and since $a$ is arbitrary this will conclude the proof.

Small merges can only occur when $t\in \Gamma_{q,r}$ for some $1\leq q<r\leq k$. And a.s. $\Gamma_{q,r}$ is locally finite on $(0,1)$ as a Poisson point process of rate $\rho$ which is locally finite on $(0,1)$. 

For the large merges, note {that} for a jump to occur at time $t$ there must exist $j\in \N$ with $\#S^\Theta(t,j)\cap[k]\geq 2$. This occurs for every time $t$ with probability upperbounded by $k^2\sum_{j=1}^\infty \theta_j(t)^2$. Thus thanks to \eqref{eq:pas de coalescence avec proba non nulle}, a.s. the number of possible times $t \in[a, \max_{1 \le r \le k} H_r ]$ where large merges can occur is finite. 
\end{proof}
\begin{lemma} A.s. for every $0\leq t\leq 1$ and  $1\leq r \leq k <k'$ we have $\Part_r^{k'}(t)\cap[1,\dots, k]=\Part_r^{k}(t)$.
\end{lemma}
\begin{proof} Since the number of jumps is locally finite on $(0,1)$ the result follows by a simple (but tedious) downward induction on $t$. We omit the details. 
\end{proof}
The above consistency condition allows us to define for every $r\in \N$, $0\leq t \leq 1$, $\Part_r(t)\coloneqq\bigcup_{k\in \N}\Part_r^k(t)$.

Also, note that below $\min_{1 \le r \le k} H_r$, there are no birth times anymore for the process $((\Part_r^k(t))_{1\leq r \leq k})_{0\leq t\leq 1}$, so that the number of $r\in [k]$ such that $\Part^k_r(t) \neq \emptyset$ can not increase when $t$ decreases {below $\min_{1 \le r \le k} H_r$}. Therefore, the process $((\Part_r^k(t))_{1\leq r \leq k})_{0\leq t\leq 1}$ undergoes only a finite number of jumps in $[0,1]$. The genealogy of the $k$ particles is then described by a finite tree $\T_k$ whose leaves, denoted by $V_1,\ldots,V_k$ correspond to the births of each particle and whose internal vertices correspond to the coalescences between clusters. Each vertex $v$ of $\T_k$ is associated with a time $t_v$ corresponding to the birth time if $v$ is a leaf or the coalescence time {if} $v$ is an internal vertex. We equip this tree with the distance induced by the times of birth and coalescence: for every pair $u,v$ of vertices of $\T_k$, the distance $d(u,v)$ between $u$ and $v$ is given by
$$
d(u,v) = t_u+t_v-2c(u,v),
$$
where $c(u,v)$ is the first time at which the clusters of $u$ and $v$ coalesce. Note that if $V_q,V_r$ are leaves of $\T_k$ corresponding to the particles born at times $H_q,H_r$ for $1 \le q<r\le k$, then 
\begin{equation}\label{eq lien distance coalescent continu}
d(V_q,V_r)= H_q+H_r - 2\max \{t \in [0,1], \ \exists m\in [k], \ q,r \in \mathcal{P}^k_m(t) \}.
\end{equation}
Since the above equation can also be written
$$
d(V_q,V_r)= H_q+H_r - 2\max \{t \in [0,1], \ \exists m\ge1,  \ q,r \in \mathcal{P}_m(t) \},
$$
the random variable $d(V_q,V_r)$ does not depend on $k$.
\subsection{Leaf-tightness} \label{sec:continuum_leaf-tight}

While the previous section defines the matrix of distances between random vertices of the limit trees, it is not enough to properly define the limit. This difference is quite subtle and has been extensively studied by Aldous in \cite{Ald93}. As recalled in Appendix \ref{sec:leaf}, {more precisely by applying Proposition \ref{B.1} to the increasing sequence of the random trees $(\T_k,d)$ for $k\ge 1$ defined in the previous subsection, equipped with the uniform probability measure on the leaves $\{V_1, \ldots, V_k\}$ and with the distance $d$ defined in \eqref{eq lien distance coalescent continu},} to properly define the {measured real tree $\T(\nu, \rho , \Theta)$ as the (increasing) limit when $k\to \infty$ of the $\T_k$'s when $k\to \infty$,} it is enough to prove the following weak leaf tightness criterion: 
\[ \forall \delta>0, \qquad \lim_{k\to \infty} \proba{d(V_{k+1},\{V_1,V_2,\dots, V_k\})>{5}\delta}=0.\]
Since $H_{k+1}$ (the height of $V_{k+1}$) is sampled according to the probability measure $\nu$ which satisf{ies} $\nu(\{0\})=0$, we may assume for some $a>0$ that $a<H_{k+1}<a+\delta$. Also, since, under \ref{hyp:tightGP}, $\nu$ has support $[0,1]${,} by taking $k$ large enough, we may assume that $a-\delta<H_1,H_2,\dots, H_k<a$. ({In} detail, by considering enough vertices, we can have among them as many as we want {with probability arbitrarily close to $1$} at height between $a-\delta$ and $a$.) In other words, it suffices to check that for every $0<a\leq 1$ and $\delta>0$,
\[ \lim_{k\to \infty}  \proba{d(V_{k+1},\{V_1,V_2,\dots, V_k\})>{5}\delta \middle | H_{k+1}\in [a,a+\delta], \forall 1\leq i \leq k,\, \,  H_i\in[a-\delta,a]}=0.\]
Then we write $\probak{\cdot}=\P(\cdot|H_{k+1}\in [a,a+\delta], \forall 1\leq i \leq k, \, \, H_i\in[a-\delta,a])$. Define similarly the associated expectation $\E^k$ for later. And we recall that for all $i\neq j$, the random variable $c(V_i,V_j)$ is the coalescence time between (the genealogies of) $V_i$ and $V_j$, that is the height of their most recent common ancestor. So that by \eqref{eq lien distance coalescent continu}, to prove the above it suffices to show
\begin{equation} \lim_{k\to \infty}  \probak{\forall 1\leq i \leq k, \ c(V_i,V_{k+1})<a-2\delta}=0.\label{eq:desired_limit_leaf_tight}\end{equation}

To that end we adapt an argument from Section 6.5 of \cite{BBKKbis23+}. First for every $1\leq i\leq k$, $x\geq 0$, consider the event
\[ E_{i,k}(x)\coloneqq \{H_i>x, \forall j\in [k]\backslash \{i\}, \  c(V_i,V_j)<x \}, \]
that is the event that $V_i$ has height at least $x$ and has not coalesced with any other vertices among $V_1,\dots, V_k$ above height $x$. Then let
\[ N_k(x)\coloneqq\#\{1\leq i \leq k, E_{i,k}(x) \}. \]
To prove \eqref{eq:desired_limit_leaf_tight}, we rely on the simple but crucial remark: under $H_{k+1}>H_1,\dots, H_k$, $V_{k+1}$ has the most chance among $V_1,V_2,\dots, V_{k+1}$ to have coalesced with one of the others. In other words, 
\[ \probak{E_{k+1,k+1}(a-2\delta)} \leq \frac{1}{k+1}\sum_{i=1}^{k+1} \probak{E_{i,k+1}(a-2\delta)} = \frac{1}{k+1} \E^k\left [N_{k+1}(a-2\delta)\right ].\]
Note that the left term above is exactly the probability in \eqref{eq:desired_limit_leaf_tight}, so to define the limit it suffices to prove:
\begin{lemma} \label{lem:end_lim_leaf-tight} For any choice of $a,\delta>0$ as $k\to \infty$, $\E^k[N_{k+1}(a-2\delta)]=o(k)$.
\end{lemma}
\begin{proof} 
For each small interval $[t,t+dt]$, given $N_{k+1}(t+dt)$ {we look} at the expected number of vertices in $V_1,\dots, V_{k+1}$ that have not yet merged at time $t+dt$ that merge with another vertex between time $t$ and $t+dt$. As in the definition of our coalescent, we distinguish two cases: either there is a large merge at time $t$ or there is a small merge between time $t$ and $t+dt$. The second case is easy to count as the merging occurs by pair at rate given by $\rho$. For the first case we look for every $V_r$ not merged at the value of $\Theta_{t,r}$ and check if there is no other $V_{r'}$ not yet merged that {has} the same value. This, together with the fact that $N_{k+1}(s)\ge N_{k+1}(t)$ for all $s\in [t,t+dt]$ {for all $t\le t+dt$ such that $a-2\delta \le t \le t+dt \le a-\delta$}, gives us the next formula: for every $a-2\delta<t\le t+dt<a-\delta$,
\begin{align*} \E^k [N_{k+1}(t+dt)-N_{k+1}(t)] 
\geq \E^k \left[\sum_{j\geq 1:\theta_j(t)>0}  N_{k+1}(t)\theta_j(t)(1-(1-\theta_j(t))^{N_{k+1}(t)-1}) +2\binom{N_{k+1}(t)}{2}\rho(t,t+dt)\right].
\end{align*}
Using the fact that for all $x\in (0,1),y>0$ with $xy>1$ we have $(1-x)^y\leq 1/e$, this implies that
 \[ \E^k[N_{k+1}(t+dt)-N_{k+1}(t) ]
 \geq \E^k\left[\sum_{j\geq 1: \theta_j(t)>1/(N_{k+1}(t)-1)} \theta_j(t) \left(1-\frac{1}{e} \right){N_{k+1}(t)}  \right]+2\E^k\left[\binom{N_{k+1}(t)}{2}\right]\rho(t,t+dt), \]
Then, by integrating between $a-2\delta$ and $a-\delta$, and using  $0\leq N_{k+1}\leq k+1$ for the left hand side, we get:
\[ k+1 \geq \sum_{\substack{ t\in (a-2\delta,a-\delta), j\geq 1}} \theta_j(t) \left(1- \frac{1}{e}\right)\E^k\left [   {N_{k+1}(t)}\1_{\theta_j(t)>1/(N_{k+1}(t)-1)} \right]+\int_{a-2\delta}^{a-\delta} 2\E^k \left [\binom{N_{k+1}(t)}{2}\right ]\rho(\mathrm{d}t). \]
Next, by fixing $\e>0$ and considering only the case where $N_{k+1}(t)>\e k$, we obtain 
\[ k+1 \geq \sum_{\substack{ t\in (a-2\delta,a-\delta), j\geq 1}} \theta_j(t) \left(1-\frac{1}{e}\right){\e k}\probak{N_{k+1}(t)>\e k}
\1_{\theta_j(t)>1/(\e k-1)}
+\int_{a-2\delta}^{a-\delta}(\e k-1)^2 \probak{N_{k+1}(t)>\e k}\rho(\mathrm{d}t). \]
{Then,} using that for all $t\geq a-2\delta$, we have $N_{k+1}(t)\geq N_{k+1}(a-2\delta)$, {we get}
\[ k+1 \geq \probak{N_{k+1}(a-2\delta )>\e k}\left (  \left(1-\frac{1}{e}\right){\e k} \sum_{\substack{ t\in (a-2\delta,a-\delta), j\geq 1 \\\theta_j(t)>1/(\e k-1)  }} \theta_j(t)
 +(\e k-1)^2\rho([a-2\delta,a-\delta]) \right ).\]
 Using the second part of \ref{hyp:tightGP}, $k$ is negligible compared to  the inside of the huge parenthesis above, so 
 \[ \probak{N_{k+1}(a-2\delta )>\e k}\cv[k] 0.\]
 Since $\e$ is arbitrary and since $N_{k+1}\leq k+1$ this implies the desired result, and so concludes the section.
\end{proof}

\section{GP convergence} \label{sec:proofGP}

\subsection{Coalescent associated with the genealogies of random vertices in $\T^n$}\label{sous-section coalescent discret}
In this section, we introduce a discrete growth-coalescent process, which describes how the genealogies of random vertices in $\T^n$ merge. In the next section, we will prove a scaling limit for this coalescent toward the similar continuous coalescent introduced in \cref{sec: coalescent continu}. At the end of this section, we  explain why this implies the scaling limit of the distance matrix of random vertices in $\T^n$, and so  Theorem \ref{th:GP}.

Beforehand, let us formally split high degrees and small degrees.
\begin{lemma}\label{lemme epsilon n} There exists a sequence $(\e_n)_{n\ge0}$ decreasing to $0$ slowly enough so that for all $n\ge 0$, there exists a bijection $I_n : \{ t,\ \theta_1(t)>\e_n \text{ and } \e_n <t <1-\e_n \} \to 
\{ i, \ d^n_{i,1}/D^n_i>\e_n \text{ and }
\e_n < i/n < 1-\e_n\}$ 
 such that for all $\e>0$,
\begin{equation}\label{eq:cv atomes L2}
\sum_{j \ge 1, t\in (\e, 1-\e)}
\left(
\frac{d^n_{I_n(t), j}}{D^n_{I_n(t)}} - \theta_j(t)
\right)^2
\1_{\theta_j(t)>\e_n }
\cv[n]
0
\end{equation}
and such that for all $t\ge0$ such that $\theta_1(t)>0$,
\begin{equation}\label{eq: cv I sur n vers t}
I_n(t)/n \cv[n] t.
\end{equation}
\end{lemma}
\begin{proof}
We construct the sequence $\e_n$ as follows. For all $k \ge 1$, let $\eta_k \in (2^{-k-1},2^{-k})$ such that $\eta_k \not \in \{\theta_j(t); \ t \in (0,1), j\ge 1\}$. By \ref{hyp:atomes} and \ref{hyp:splitatomes}, there exists $n_k\ge 1$ such that for all $n\ge n_k$, there exists a bijection $I_{k,n}: \{t \in (\eta_k,1-\eta_k), \ \theta_1(t)> \eta_k\} \to \{i, \ d^n_{i,1}/D^n_i > \eta_k \text{ and } \eta_k < i/n < 1-\eta_k \}$ such that for all $t \in (\eta_k, 1-\eta_k)$, for all $j\ge 1$ such that $\theta_j(t)>\eta_k$,
$$
\left\vert \frac{d^n_{I_{k,n}(t),j}}{D^n_{I_{k,n}(t)}} - \theta_j(t) \right\vert \le \eta_k \qquad \text{and} \qquad 
\left\vert \frac{I_{k,n}(t)}{n} -t \right\vert \le \eta_k.
$$
Then it suffices to take $\e_n \coloneqq \min\{\eta_k; \ k\ge 1 \text{ and } n_k \le n\}$ and define $I_n$ as the $I_{k,n}$ where $k$ is the largest integer such that $n_k \le n$. Then \eqref{eq: cv I sur n vers t} holds clearly. For \eqref{eq:cv atomes L2}, note that for all $\e>0$,
$$
\sum_{j \ge 1, t\in (\e, 1-\e)}
\left(
\frac{d^n_{I_n(t), j}}{D^n_{I_n(t)}} - \theta_j(t)
\right)^2
\1_{\theta_j(t)>\e_n }
\le \sum_{j \ge 1, t\in (\e, 1-\e)} \e_n^2 \1_{\theta_j(t)>\e_n }
\le \sum_{j \ge 1, t\in (\e, 1-\e)} \e_n^2\wedge \theta_j(t)^2
\cv[n] 0,
$$
where the convergence holds by dominated convergence thanks to \eqref{eq:pas de coalescence avec proba non nulle}.
\end{proof}

Then, let us prove a direct consequence of \ref{hyp:splitatomes} which will be very convenient:
\begin{lemma}\label{lemme profil tend vers l'infini}
    We have for all $\e \in (0,1)$,
$$
\min_{ \e n \le i \le (1-\e)n} D^n_i \cv[n] \infty.
$$
\end{lemma}
\begin{proof}
Let $\e\in (0,1)$. Let us prove that for all $n$ large enough for all $\e n < i < (1-\e) n$, we have $D^n_i \ge 1/ \e$. For all $i \le \h^n -1$, if $d^n_{i,1} \le \e D^n_i$, then, since $d^n_{i,1}\ge 1$, we deduce that $D^n_i \ge 1/\e$.

Besides, note that for all $n$ large enough, if $d^n_{i,1} > \e D^n_i$ for some $\e n < i <(1-\e) n$, then by \ref{hyp:splitatomes} we have $d^n_{i+1, 1} \le \e D^n_{i+1}$. But we also know that $D^n_i d^n_{i+1,1} \ge D^n_{i+1}$
since there are exactly $D^n_i$ vertices of $\T^n$ at height $i+1$, so that 
$$
D_{i+1}^n \le d^n_{i+1,1} D^n_i \le  \e  D^n_{i+1} D^n_i.
$$
Finally, we also get $D^n_i \ge 1/ \e$ by diving both sides by $D_{i+1}^n$ (that can be done since since by \ref{hyp:naissance}, $\liminf \h^n/n \ge 1$ and so $D_{i+1}^n\ge 1$ for all $n$ large enough and for all $i< (1-\e) n $).
\end{proof}

Now, {in the rest of this subsection, let us fix $k \in \N$}. To prove the convergence of the rescaled matrix of distances between $k$ vertices $V^n_1,\ldots, V^n_k$ taken uniformly at random, we introduce the associated growth-coalescent process and prove that it satisfies a scaling limit towards the coalescent process defined in Section \ref{sec: coalescent continu}. 

In order to get some additional independence, we add some vertices at each height to the ancestors of $V^n_1,\ldots, V^n_k$ in order to have a constant number of vertices at each height {(this will facilitate the computation of the coalescence probabilities)}. For all $i> \max_{r \in [k]} \h(V^n_r)$:
\begin{itemize}
\item When $D^n_{i-1} \ge k$, let $(V^n_{i,j})_{1\leq j \leq k}$ be distinct vertices in $\{ (i,j);  \ j \le D^n_i \}$ taken uniformly at random.
\item When $D^n_{i-1} < k$, we set $V^n_{i,r} = (i, 1)$ for all $r \in [k]$. 
\end{itemize} Then, we define $V^n_{i,1}, \ldots, V^n_{i,k}$ for $i \le \max_{r \in [k]} \h(V^n_r)$ by downward induction on $i$. Given $(V^n_{i+1,j})_{1\leq j \leq k}$ define $(V^n_{i,j})_{1\leq j \leq k}$ as follows:
\begin{itemize}
    \item[] \textbf{Position of $V^n_{1}$,\dots, $V^n_k$:} For all $r \in [k]$, if $\h(V^n_r) = i$, then we set $V^n_{i,r}= V^n_r$.
    \item[] \textbf{Position of the ancestors: } For all $p \ge 1$, for all $1 \le r_1 < \ldots <r_p \le k$, if the vertices $V^n_{i+1, r_1}, \ldots V^n_{i+1,r_p}$ are the vertices in $\{V^n_{i+1,1}, \ldots, V^n_{i+1,k} \}$ which are attached to a vertex $(i,j)$ in $\T^n$, and if one of them is an ancestor of $V^n_r$ for some $r \in [k]$, then:
    \begin{itemize}
        \item If $(i,j) = V^n_{r'}$ for some $r' \in [k]$, we set as explained before $V^n_{i,r'} = V^n_{r'}$.
        \item Otherwise, if $(i,j) \not \in \{V_1^n, \ldots, V^n_k\}$, let $r$ be the smallest integer in $[k]$ such that $r \in \{r_1, \ldots, r_p \}$ and such that $V^n_{i+1,r}$ is the ancestor of $V^n_{r'}$ for some $r' \in [k]$. 
        We then set $V^n_{i,r} = (i,j)$.
    \end{itemize}
    \item[] \textbf{List completion:} When $D^n_{i-1} \ge k$, we complete the list by taking distinct vertices uniformly at random among the vertices $\{(i,j)\}_{1\leq j\leq D^n_{i-1}}$ not yet taken. By convention, when $D^n_{i-1} <k$, all the $V^n_{i,r}$ for $r \in [k]$ not yet defined are equal to $(i,1)$. By Lemma \ref{lemme profil tend vers l'infini}, this latter case is excluded on any interval $[\e n , (1-\e)n ]$ for every $\e>0$ for every $n$ large enough.
\end{itemize}

We look at the fathers of $V^n_{i+1,1}, \ldots ,V^n_{i+1,k}$  in $\T^n$. If $V^n_{i+1,q}$ and $V^n_{i+1,r}$ for some $q<r \in [k]$ have the same father, then we say that they coalesce at this vertex. {Let $(\e_n)_{n\ge 0}$ be as in Lemma \ref{lemme epsilon n}.} For all $q<r \in [k]$ and $i \ge 0$, let
\begin{equation}\label{eq indicatrice petite coalescence}
X^n_{q,r,i} = \1_{V^n_{i+1,q}, V^n_{i+1,r} \text{ coalesce in a vertex of degree at most } \e_n D^n_i}.
\end{equation}
For all $n,i\ge 0$, and $r \in [k]$, let $\Theta_{i,r}^n$ be the integer $j$ such that the vertex $V^n_{i+1,r}$ is attached to the vertex of degree $d^n_{i,j}$ when $d^n_{i,j}> \e_n D^n_i$ and $\Theta_{i,r}^n=0$ otherwise. We set for all $i,j \ge 1$,
$$
S^{n,k}_{i,j} = \{ r \in [k], \ \Theta^n_{i,r}=j\}.
$$

We are now in position to define the growth-coalescent process. Let $((\mathcal{P}^{n,k}_r(i))_{r \in [k]})_{ i\ge 0}$ be the process defined as follows. We set for all $i >\max_{r \in [k]} \h(V^n_r)$, and for all $r \in [k]$, the initial values $\mathcal{P}^{n,k}_r(i) = \emptyset$. Then, for all $i\ge 0$,
\begin{itemize}
\item[] \textbf{Birth:} If $i=\h(V^n_r)$ for some $r \in [k]$, then $\Part_r^{n,k}(i)=\{r\}$;
\item[] \textbf{Small merge:} If $X^n_{q,r,i}=1$ for some $1 \le q<r\le k$ then whenever $\Part_q^{n,k}(i+1)\neq \emptyset$ and $\Part_{r}^{n,k}(i+1)\neq \emptyset$ we have $\Part_q^{n,k}(i)=\Part_q^{n,k}(i+1)\cup \Part_{r}^{n,k}(i+1)$ and $\Part_r^{n,k}(i)=\emptyset$;
\item[] \textbf{Large merge:} If $\# S^{n,k}_{i,j} \ge 2$ for some $j\ge 1$, we write $m^{n,k}(i,j)$ for the smallest integer $r\in S^{n,k}_{i,j}$ such that $\Part_r^{n,k}(i+1)\neq \emptyset$. We set:
\[ \forall r \in S^{n,k}_{i,j} \setminus \{ m^{n,k}(i,j)\}, \quad \Part_r^{n,k}(i)=\emptyset, \quad \text{and} \quad  \Part_{m^{n,k}(i,j)}^{n,k}(i)=\bigcup_{r\in S^{n,k}_{i,j}} \Part_{r}^{n,k}(i+1);\]
\end{itemize}
and for the other $r \in [k]$ we set $\Part_r^{n,k}(i)= \Part_r^{n,k} (i+1)$.

This discrete growth-coalescent is well-defined and gives the genealogy of $V^n_1, \ldots, V^n_r$ in $\T^n$. 
\begin{proposition}
    With high probability, the discrete growth coalescent process $((\mathcal{P}^{n,k}_q(i))_{q \in [k]})_{i\ge 0}$ is well-defined. Moreover, for all $0 \le i \le \h^n$, {for all $q \in [k]$,} the set $\mathcal{P}^{n,k}_q(i)$ is the set of $r \in [k]$ such that the vertex $V^n_{i,q}$ is an ancestor of $V^n_r$.
\end{proposition}
\begin{proof}
    The {first part of the statement} is straightforward after checking that for all $\e\in (0,1)$, with high probability, for all $\e n \le i \le (1-\e ) n$, at most two vertices coalesce in vertices of degree at most $\e_n D^n_i$. This is a direct consequence of Lemma \ref{lemme distance en variation totale aux Bernoulli} which is stated and proven in the next subsection {and of the definition of $X^n_{i,k}$ just above the said lemma}. One can then use that by \ref{hyp:naissance},
    $$
    \lim_{\e \to 0} \liminf_{n\to \infty} \P\left( \forall r \in [k], \ \e n \le \h(V^n_r) \le (1-\e) n \right) =1.
    $$
    {Next, let us turn to the second point of the statement. Note that by definition of the vertices $V^n_{i,r}$'s, when there is a small merge at height $i\ge 0$, the father of the two vertices $V^n_{i+1,q}$ and $V^n_{i+1,r}$ for $q<r$ which coalesce is $V^n_{i,q}$. Similarly, when there is a large merge for some $j\ge 1$, the father of the vertices $V^n_{i+1, r}$'s for $r \in S^{n,k}_{i,j}$ is the vertex $V^n_{i, m^{n,k}(i,j)}$. This entails the second point of the proposition by taking the definition of $((\Part^{n,k}_r(i))_{r \in [k]})_{i\ge 0}$ into account and then applying a downward induction.
    }
\end{proof}

We then state the scaling limit result towards the continuous time growth-coalescent process defined in Section \ref{sec: coalescent continu}. Its proof is postponed to the next subsections.

\begin{proposition}\label{prop: cv coalescent Skorokhod}
    We have the convergence
    $$
    \left(\left(\Part^{n,k}_r(\lceil n(1-t) \rceil\right)_{r \in [k]}\right)_{t\in [0,1]}
    \cvloi[n]
    \left(\left( \Part^k_r(1-t)\right)_{r \in [k]}\right)_{t \in [0,1]}
    $$
    for the $J_1$-Skorokhod topology.
\end{proposition}
The convergence in the sense of Gromov-Prokhorov readily follows from the above proposition:
\begin{proof}[Proof of Theorem \ref{th:GP} using Proposition \ref{prop: cv coalescent Skorokhod}]
    {Let $V^n_1, \ldots, V^n_k$ be i.i.d.\@ uniform random vertices of $\T^n$ (conditionally on $\T^n$) and let $V_1, \ldots, V_k$ be the leaves of the tree $\T_k\subset \T(\nu, \rho,\Theta)$ defined in p.6. Note that the points $V_1,\ldots, V_k$ can also be seen as i.i.d.\@ random points in the metric measured space $\T(\nu, \rho,\Theta)$ since the measured real tree $\T(\nu, \rho,\Theta)$ is defined as the (increasing) limit of the $\T_k$'s as $k\to \infty$ in Subsection \ref{sec:continuum_leaf-tight}.} Thanks to Lemma \ref{equivGP}, it suffices to prove the convergence for all $k\ge 1$
    \begin{equation}\label{eq cv matrice des distances}
        \left( d^n(V^n_q,V^n_r)/n \right)_{1 \le q < r \le k} 
        \cvloi[n]
        \left( d(V_q,V_r) \right)_{1 \le q < r \le k} ,
    \end{equation}
    where $d^n$ is the graph distance on $\T^n$. Let $k\ge 1$. Let $q<r \in [k]$. By construction of $\T^n$ and by definition of the processes $\Part^{n,k}_m$ for $m \in [k]$, we have
    $$
    {d}^n(V^n_q,V^n_r) 
    = \h(V^n_q) +\h(V^n_r)-2 \max \{i \ge 0, \ \exists m \in [k], \ q,r \in \Part^{n,k}_m(i) \}.
    $$
    Therefore, by \eqref{eq lien distance coalescent continu}, in order to show \eqref{eq cv matrice des distances}, it is enough to check the convergence
    $$
        \frac{1}{n}\max \{i \ge 0, \ \exists m \in [k], \ q,r \in \Part^{n,k}_m(i) \}
        \cvloi[n]
        \max \{ t \ge 0, \ \exists m \in [k], \ q, r \in \Part_m^k(t) \},
    $$
    jointly with 
    \begin{equation}\label{eq cv hauteurs unif}
    ((\h(V^n_r)/n)_{r \in [k]} \cvloi[n] (H_r)_{r \in [k]}.
    \end{equation} 
    But note that $\h(V^n_r) = \max \{ i\ge 0, \ r \in \Part_r^{n,k}(i) \}$ and $H_r = \max\{t \ge 0, \ r \in \Part^k_r(t)\}$. Thus, the two convergences stem from Proposition \ref{prop: cv coalescent Skorokhod}.
\end{proof}

The rest of this section is dedicated to the proof of Proposition \ref{prop: cv coalescent Skorokhod}.

\subsection{Convergence of the coalescences between random vertices at the same height}

In this subsection, we prove Proposition \ref{prop: cv coalescent Skorokhod}. We first prove some useful lemmas.

A consequence of \ref{hyp:coalescence}, (\ref{eq:cv atomes L2}) and \eqref{eq: cv I sur n vers t} is the {vague} convergence of measures on $(0,1)$
\begin{equation}\label{eq:cv rho epsilon n}
n\sum_{j\ge 1} \frac{\binom{d^n_{nt,j}}{2}}{\binom{D^n_{nt}}{2}} \1_{d^n_{nt,j}\le \e_n D^n_{nt}} \mathrm{d} t
\cv[n]
\rho.
\end{equation}
We set for all $n,i\ge 0$,
$$
p_{n,i}\coloneqq \sum_{j\ge 1} \frac{\binom{d^n_{i,j}}{2}}{\binom{D^n_{i}}{2}} \1_{d^n_{i,j}\le \e_n D^n_{i}}.
$$
Then Equation (\ref{eq:cv rho epsilon n}) entails that for all $\e>0$, we have \begin{equation}\label{eq p n i petit}
\max_{\e n \le i \le (1-\e)n} p_{n,i} \cv[n] 0
\end{equation}
since $\rho$ has no atoms. Another consequence of (\ref{eq:cv rho epsilon n}) is the following lemma:
\begin{lemma}\label{lemme Bernoulli Poisson}
    For all $k\ge 2$, for any integer $n$ large enough, for $i\ge 0$, let $B^n_{i}$ be independent Bernoulli random variables of parameter $p_{n,i}$. Then we have the {vague} convergence of measures on $(0,1)$ in distribution
    $$
    \sum_{1\le i <n} B^n_{i} \delta_{i/n} 
    \cvloi[n]
    \Pi_{\rho},
    $$
    where $\Pi_\rho$ is a Poisson random measure of intensity $\rho$. 
\end{lemma}
\begin{proof}
    It is a consequence of the Poisson paradigm. By the independence of the $B^n_{i}$'s and the fact that $\rho$ has no atoms, it suffices to show that for all $0<s<t<1$, for all $x\in \R$, we have
    $$
    \E \left[
    e^{\mathrm{i}x \sum_{ns \le j \le nt} B^n_{j}}
    \right]
    \cv[n]
    \exp(\rho([s,t])(e^{\mathrm{i} x }-1))
    ,
    $$
    which is straightforward since
    $$
    \E \left[
    e^{\mathrm{i}x \sum_{ns \le j \le nt} B^n_{j}}
    \right]
    = \prod_{ns \le j \le nt} \left(
    p_{n,j}e^{\mathrm{i} x}
    +1-p_{n,j}
    \right)
    $$
    and then using \eqref{eq:cv rho epsilon n} and \eqref{eq p n i petit}.
\end{proof}

Let us focus on the times of small merges. For all $n,i\ge 0$, let $X^n_{i,k}$ be the number of pairs $\{r,s\}$ for $r<s \in[k]$ such that $V^n_{i+1,r}$ and $V^n_{i+1,s}$ coalesce in a vertex of degree at most $\e_n D^n_i$, i.e.
$$
X^n_{i,k} = \sum_{1\le r < s\le k} \1_{V^n_{i+1,r}, V^n_{i+1,s} \text{ coalesce in a vertex of degree at most } \e_n D^n_i}.
$$

\begin{lemma}\label{lemme distance en variation totale aux Bernoulli}
    For all $n\ge 0$ large enough, for all $k\ge 2$, let $(B_{i,k}^{n})_{i\ge 0}$ a family of independent Bernoulli random variables of parameters $\binom{k}{2}p_{n,i}$. Then for all $\e>0$,
    $$
    d_{\mathrm{TV}}\left( 
    \left(X^n_{i,k}\right)_{\e n < i < (1-\e)n}, 
    \left( B^{n }_{i,k}\right)_{\e n < i < (1-\e)n}
    \right)
    \cv[n] 0,
    $$
    where $d_\mathrm{TV}$ denotes the total variation distance.
\end{lemma}
\begin{proof}
Let $\e>0$. By definition of $X^n_{i,k}$, it is enough to show that there exists a coupling between $X^n_{i,k}$ and $B^n_{i,k}$ such that
\begin{equation}\label{eq:somme des probas X different de B}
    \sum_{\e n \le i \le (1-\e) n } \P(X^n_{i,k} \neq B^n_{i,k}) \cv[n] 0.
\end{equation}
Note that by Lemma \ref{lemme profil tend vers l'infini}, for $n$ large enough, for all $\e n < i < (1- \e) n$, we have $D^n_i \ge k$, so that the vertices $V^n_{i+1,1}, \ldots, V^n_{i+1,k}$ are distinct. First of all, we see that
\begin{equation}\label{eq:espX}
    \E \left[X^n_{i,k}\right] =
     \sum_{j\ge 1} \frac{\binom{d^n_{i,j}}{2} \binom{D^n_i -2}{k-2}}{\binom{D^n_i}{k}} \1_{d^n_{i,j} \le \e_n D^n_i}\\
     =
     \binom{k}{2} p_{n,i}.
\end{equation}
Besides, we have
\begin{align}
    \E\left[
    (X^n_{i,k})^2
    \right]
    =& \E\left[X^n_{i,k}\right] \label{espcarre1} \\
    &+
    6 \E\left[
    \sum_{1 \le r_1 < r_2 < r_3 \le k}
    \1_{V^n_{i+1,r_1}, V^n_{i+1,r_2}, V^n_{i+1,r_3} \text{ coalesce in a vertex of degree at most } \e_n D^n_i} 
    \right] \label{espcarre2}\\
    &+ \E \left[\sum_{\substack{1\le r_1 < s_1 \le k\\ 1 \le r_2 < s_2 \le k}} \1_{\{r_1,s_1\}\cap \{r_2,s_2\} = \emptyset} \1_{\substack{V^n_{i+1,r_1}, V^n_{i+1,s_1} \text{ coalesce in a vertex of degree at most } \e_n D^n_i \\\text{and }V^n_{i+1,r_2}, V^n_{i+1,s_2} \text{ coalesce in a vertex of degree at most } \e_n D^n_i}} \right]\label{espcarre3}
\end{align}
The first term (\ref{espcarre1}) is $\binom{k}{2} p_{n,i}$ by (\ref{eq:espX}). The second term (\ref{espcarre2}) is
\begin{equation}\label{eq: majoration espcarre2}
\sum_{j\ge 1} \frac{\binom{d^n_{i,j}}{3} \binom{D^n_i-3}{k-3}}{\binom{D^n_i}{k}} \1_{d^n_{i,j} \le \e_n D^n_i} 
= \1_{D^n_i\ge 3}\sum_{j\ge 1} \frac{d^n_{i,j}-2}{3} \frac{k-2}{D^n_i-2} \frac{\binom{d^n_{i,j}}{2} \binom{D^n_i-2}{k-2}}{\binom{D^n_i}{k}} \1_{d^n_{i,j} \le \e_n D^n_i} 
\le \e_n \frac{k-2}{3} \binom{k}{2} p_{n,i},
\end{equation}
while the third term (\ref{espcarre3}) is 
\begin{align}
    &\sum_{\substack{j_1,j_2\ge 1 \\ j_1 \neq j_2}} \1_{d^n_{i,j_1}, d^n_{i,j_2} \le \e_n D^n_i}\frac{\binom{d^n_{i,j_1}}{2}\binom{d^n_{i,j_2}}{2}\binom{D^n_i-4}{k-4}}{\binom{D^n_i}{k}}
    + 6 \sum_{j\ge 1} \frac{\binom{d^n_{i,j}}{4} \binom{D^n_i - 4}{k-4}}{\binom{D^n_i}{k}} \1_{d^n_{i,j} \le \e_n D^n_i}\notag\\
    &=\1_{D^n_i\ge 4} \binom{k}{2}\binom{k-2}{2}
    \sum_{\substack{j_1,j_2\ge 1 \\ j_1 \neq j_2}}
    \frac{\binom{d^n_{i,j_1}}{2}}{\binom{D^n_i}{2}}
    \frac{\binom{d^n_{i,j_2}}{2}}{\binom{D^n_i-2}{2}}
    \1_{d^n_{i,j_1}, d^n_{i,j_2} \le \e_n D^n_i} +  6 \1_{D^n_i\ge 4}  \binom{k}{4} \sum_{j\ge 1} \frac{\binom{d^n_{i,j}}{4}}{\binom{D^n_i}{4}} \1_{d^n_{i,j} \le \e_n D^n_i}
    \notag\\
    &\le\1_{D^n_i\ge 4} \frac{D^n_i(D^n_{i}-1)}{(D^n_i-2)(D^n_i-3)} \binom{k}{2}\binom{k-2}{2}
    \sum_{\substack{j_1,j_2\ge 1 \\ j_1 \neq j_2}}
    \frac{\binom{d^n_{i,j_1}}{2}}{\binom{D^n_i}{2}}
    \frac{\binom{d^n_{i,j_2}}{2}}{\binom{D^n_i}{2}}
    \1_{d^n_{i,j_1}, d^n_{i,j_2} \le \e_n D^n_i} +  6 \1_{D^n_i\ge 4}  \binom{k}{4} \sum_{j\ge 1} \frac{\binom{d^n_{i,j}}{2}^2}{\binom{D^n_i}{2}^2} \1_{d^n_{i,j} \le \e_n D^n_i}
    \notag\\
    &\le
    \1_{D^n_i\ge 4}\left(\frac{D^n_i(D^n_{i}-1)}{(D^n_i-2)(D^n_i-3)}
    \binom{k}{2}\binom{k-2}{2} + 6 \binom{k}{4}\right)
    \left(\sum_{j\ge 1} \frac{\binom{d^n_{i,j}}{2}}{\binom{D^n_i}{2}}
    \1_{d^n_{i,j} \le \e_n D^n_i} \right)^2 \notag\\
    &\le 12\binom{k}{2}\binom{k-2}{2} (p_{n,i})^2. \label{eq: majoration espcarre3}
\end{align}
For all $n$ large enough, let $B^n_{i,k}$ be a Bernoulli random variable of parameter $\binom{k}{2} p_{n,i}$ such that $B^n_{i,k}\ge \1_{X^n_{i,k}\ge 1}$. One can indeed perform such a coupling given that $\P(X^n_{i,k}\ge 1) \le \E\left[X^n_{i,k}\right]=\binom{k}{2} p_{n,i} \le 1$ for $n$ large enough.
As a result,
\begin{align*}
\sum_{\e n \le i \le (1-\e) n } \P(X^n_{i,k} \neq B^n_{i,k})
&\le \sum_{\e n \le i \le (1-\e) n } \left(\P(X^n_{i,k} \neq \1_{X^n_{i,k}\ge 1})
+
\P(B^n_{i,k} \neq \1_{X^n_{i,k}\ge 1})
\right)\\
&\le \sum_{\e n \le i \le (1-\e) n } \left(
\E\left[
\left\vert 
X^n_{i,k} - \1_{X^n_{i,k}\ge 1}
\right\vert^2
\right]
+\E\left[
X^n_{i,k}
\right]
-\P(X^n_{i,k}\ge1)
\right)
\\
&= \sum_{\e n \le i \le (1-\e) n } \left(
\E\left[(X^n_{i,k})^2\right]
-\E\left[
X^n_{i,k}
\right]
\right) \\
&\le 
\sum_{\e n \le i \le (1-\e) n } \left(
\e_n \frac{k-2}{3} \binom{k}{2} p_{n,i}
+
12 \binom{k}{2}\binom{k-2}{2} (p_{n,i})^2
\right) \text{ by \eqref{eq: majoration espcarre2} and \eqref{eq: majoration espcarre3},} 
\\
&\cv[n] 0 \text{ by \eqref{eq:cv rho epsilon n} and \eqref{eq p n i petit}},
\end{align*}
hence (\ref{eq:somme des probas X different de B}) {follows}.
\end{proof}

Combining Lemmas \ref{lemme Bernoulli Poisson} and \ref{lemme distance en variation totale aux Bernoulli} we obtain the scaling limit of the point processes induced by the $X^n_{q,r,i}$'s which are defined in \eqref{eq indicatrice petite coalescence}.
\begin{corollary}\label{cor: cv petites coalescences}
    We have the {vague} convergence of measures on $(0,1)$ in distribution
    $$
    \left(\sum_{i\ge 0} X^n_{q,r,i} \delta_{i/n} \right)_{1 \le q<r \le k}
    \cvloi[n]
    \left(\Gamma_{q,r}\right)_{1 \le q<r \le k},
    $$
    where the $\Gamma_{q,r}$'s for $1 \le q<r\le k$ are independent Poisson point processes of intensity $\rho(\mathrm{d} t)$.
\end{corollary}
\begin{proof}
    By Lemma \ref{lemme distance en variation totale aux Bernoulli}, we know that with high probability, at each height we have at most one coalescence in a vertex of degree at most $\e_n D^n_i$. Moreover, conditionally on $(X^n_{i,k})_{i \ge 1}$, the indices $q<r$ of the vertices $V^n_{i+1,q}, V^n_{i+1,r}$ which coalesce at each height are uniform. For all $n,i\ge 1$ let $(Q^n_i,R^n_i)$ be independent random variables which are uniform in $\{(q,r), \ 1 \le q<r \le k\}$ and taken independently of the $B^n_{i,{k}}$'s. Then,
    $$
    d_{\mathrm{TV}}\left(
    \left(X^n_{q,r,i}\right)_{\e n <i < (1-\e) n, 1 \le q < r \le k}
    ,
    \left(\1_{(Q^n_i,R^n_i)=(q,r)}B^n_{i,{k}}\right)_{\e n <i < (1-\e) n, 1 \le q < r \le k}
    \right)
    \cv[n] 0.
    $$
    Besides, {let $\left(B^n_{q,r,i}\right)_{i\ge 1}$ for $1 \le q<r \le k$ be i.i.d.\@ copies of the family $\left(B^n_{i}\right)_{i\ge 1}$ defined in Lemma \ref{lemme Bernoulli Poisson}}. One can see that
    \begin{align*}
    &d_{\mathrm{TV}} \left(
    \left(\1_{(Q^n_i,R^n_i)=(q,r)}B^n_{i,{k}}\right)_{\e n <i < (1-\e) n, 1 \le q < r \le k}
    ,
    \left(B^n_{q,r,i}\right)_{\e n <i < (1-\e) n, 1 \le q < r \le k}
    \right)\\
    &\le \!\!\!\!\!\!\!\sum_{\e n <i < (1-\e)n} \left( \P\left(\sum_{1 \le q<r\le k} B^n_{q,r,i} \ge 2\right) + \left\vert \P\left(B^n_{i,{k}} = 1\right) - \P\left(\sum_{1 \le q<r\le k} B^n_{q,r,i} = 1\right)\right\vert + \left\vert \P\left(B^n_{i,{k}} = 0\right) - \P\left(\sum_{1 \le q<r\le k} B^n_{q,r,i} = 0\right)\right\vert\right)\\
    &= 2 \sum_{\e n < i <(1-\e) n}\P\left(\sum_{1 \le q<r\le k} B^n_{q,r,i} \ge 2\right)\\
    &\le 2\binom{\binom{k}{2}}{2} \sum_{\e n < i <(1-\e)n} p_{n,i}^2\cv[n] 0,
    \end{align*}
    where the convergence comes from \eqref{eq:cv rho epsilon n} and \eqref{eq p n i petit}.
    Thus we have the stronger result
    $$
    d_{\mathrm{TV}}\left(
    \left(X^n_{q,r,i}\right)_{\e n <i < (1-\e) n, 1 \le q < r \le k}
    ,
    \left(B^n_{q,r,i}\right)_{\e n <i < (1-\e) n, 1 \le q < r \le k}
    \right)
    \cv[n] 0.
    $$
    One concludes using Lemma \ref{lemme Bernoulli Poisson}.
\end{proof}
Now, let us study the coalescences in high degree vertices. 
Recall that in the continuous coalescent, the large merges are described by the independent random variables $\Theta_{t,r}$ of law $(\theta_j(t))_{j\ge 0}$ for $r \in [k]$ and $t$ such that $\theta_1(t)>0$ which are defined in Section \ref{sec: coalescent continu}. Recall also that $S^\Theta(t,j)\cap [k]=  \{ r \in [k], \ \Theta_{t,r} = j\} $.

By \ref{hyp:splitatomes}, we may also assume that $\e_n$ decreases slowly enough and that there exists a sequence $(\eta_n)_{n \ge 0}$ such that $n \eta_n \to \infty$ and for all $n \ge 0$, for all $(i,j),(i',j')$ with $\e_n n < i,i' < (1-\e_n) n$ such that $d^n_{i,j} > \e_n D^n_i$ and $d^n_{i',j'} > \e_n D^n_{i'}$, then we have either $i=i'$ or $\vert i-i' \vert > \eta_n n$.

\begin{lemma}\label{lem: cv grandes coalescences}
    For all $\e'>0$, we have the {vague} convergence of measures on $(0,1)\times (\N\cup \{0\})^k$ in distribution
    $$
    \sum_{1 < i \le n}
    \1_{d^n_{i,1}>\e' D^n_i} \delta_{(i/n, (\Theta^n_{i,r})_{r \in [k]})}
    \cvloi[n]
    \sum_{t\ge 0} \1_{\theta_1(t)>\e'}\delta_{(t,(\Theta_{t,r})_{r \in [k]})} .
    $$
\end{lemma}
\begin{proof}
    For all $n,i\ge 0$, $j\ge 1$, let $N^n_{i,j,k}$ be the number of vertices among $V^n_{i+1,1}, \ldots, V^n_{i+1,k}$ with father $(i,j)$. 
    Note that by Lemma \ref{lemme profil tend vers l'infini}, for $n$ large enough, for all $\e n < i < (1- \e) n $, we have $D^n_i \ge k$ so that the vertices $V^n_{i+1,1}, \ldots , V^n_{i+1,k}$ are distinct. Conditionally on the $N^n_{i,j,k}$'s for $\e n < i < (1- \e) n $ and $j\ge 1$ such that $d^n_{i,j}>\e_n D^n_i$, the $N^n_{i,j,k}$ vertices which coalesce at $(i,j)$ are chosen uniformly at random among $V^n_{i+1,1}, \ldots, V^n_{i+1,k}$. More precisely, the sets $S^{n,k}_{i,j}\coloneqq \{ r \in [k], \ \Theta^n_{i,r} = j\}$ for $j \in \{j \ge 1, \ d^n_{i,j} > \e_n D^n_i \}\cup \{0\}$ form a random partition of $[k]$ with prescribed sizes $N^n_{i,j,k}$ which is chosen uniformly at random independently for all $i\ge 1$ and conditionally on the $N^n_{i,j,k}$'s. {Moreover, due to the fact that the $\Theta_{t,r}$'s for $r\in [k]$ are i.i.d.\@, the sets $S^\Theta(t,j) \cap [k]= \{r \in [k], \ \Theta_{t,r}=j \}$ for $j\ge 1$ similarly form a random partition of $[k]$ which is chosen uniformly at random among partitions of $[k]$ with prescribed sizes $N_{t,j,k}^\Theta:=\#\{r\in [k], \ \Theta_{t,r}= j \}$'s for $j\ge 1$, conditionally on the $N_{t,j,k}^\Theta$'s.} 
    
    {Recall from Lemma \ref{lemme epsilon n} the definition of $I_n$.} The statement of the lemma is thus implied by the {following} convergence for all $t\ge 0$ such that $\theta_1(t)>0${:}
    \begin{equation}\label{eq:convergence N_IJ}
    \left(
    N^n_{I_n(t), j,k}
    \right)_{j\ge 1}
    \cvloi[n]
    (\# (S^\Theta(t,j) \cap [k]))_{j\ge 1}.
    \end{equation}
    Let $t\ge 0$ such that $\theta_1(t)>0$. Let $(m_j)_{j\ge 1}$ be a family of non-negative integers such that $\sum_{j\ge 1} m_j \le k$. Then, for $n$ large enough,
    \begin{align*}
    \P\left(
    \forall j\ge 1,
    \
    N^n_{I_n(t),j,k} = m_j
    \right)
    &=
    \left({\binom{D^n_{I_n(t)}-\sum_{j\ge 1} d^n_{I_n(t), j}\1_{\theta_j(t)>\e_n}}{k-\sum_{j\ge 1} m_j}\prod_{j\ge 1} \binom{d^n_{I_n(t),j}}{m_j}}\right){\binom{D^n_{I_n(t)}}{k}}^{-1}
    \\
     &\cv[n] \left(
    \frac{\theta_0(t)^{k-\sum_{j\ge 1} m_j}}{(k-\sum_{j\ge 1} m_j)!}
    \prod_{j\ge 1} \frac{\theta_j(t)^{m_j}}{m_j!}
    \right)
    {k!}=\P\left(
    \forall j \ge 1, \ \# (S^\Theta(t,j) \cap [k])= m_j
    \right),
    \end{align*}
    by the convergence \eqref{eq:cv atomes L2}.
\end{proof}

\begin{proof}[Proof of Proposition \ref{prop: cv coalescent Skorokhod}]
    Since the processes take their values in a finite set and since the number of jumps of $(( \Part^k_r(t))_{r \in [k]})_{t \in [0,1]}$ is finite {thanks to Lemma \ref{lemme coalescent continu bien defini}}, it suffices to prove the convergence of the times of jumps and of the jumps.

    Actually, by definition of the two processes, it suffices to prove the joint convergence of the heights of the uniform vertices \eqref{eq cv hauteurs unif}, of the small coalescences on $(0,1)$
    \begin{equation}\label{eq cv petites coalescences}
    \left(\sum_{ 1\le i < n}X^n_{q,r,i} \delta_{i/n}\right)_{1 \le q<r \le k}\cvloi[n] \left(\Gamma_{q,r}\right)_{1 \le q<r\le k}
    \end{equation}
    and of the coalescence at vertices of high degrees on $(0,1)$
    \begin{equation}\label{eq cv grandes coalescences}
    \forall \e' >0,\qquad \sum_{1\le i < n} \1_{d^n_{i,1} \ge \e' D^n_i} \delta_{(i/n, (\Theta^n_{i,r})_{r\in [k]})}\cvloi[n] \sum_{t \in (0,1)} \1_{\theta_1(t)>\e'} \delta_{(t,(\Theta_{t,r})_{r \in [k]})}.
    \end{equation}
    {The fact that we can restrict our attention to the degrees $d^n_{i,1} \ge \e' D^n_i$ for $\e'$ arbitrarily small instead of $d^n_{i,1} \ge \e_n D^n_i$ comes from the fact that for all $0<x<y<1$,
    \[\sum_{x\le t\le y} \sum_{j\ge 1} \theta_j(t)^2 \1_{\theta_j(t)\le \e'} \mathop{\longrightarrow}\limits_{\e' \to 0} 0\]
    due to \eqref{eq:pas de coalescence avec proba non nulle} and from \eqref{eq:cv atomes L2} which implies that for all $0<x<y<1$,
    \[
    \sup_{n\ge 1} \sum_{xn \le i \le yn} \sum_{j\ge 1} \left( \frac{d^n_{i,j}}{D^n_i}\right)^2 \1_{\e_n D^n_i \le d^n_{i,j} \le \e' D^n_i} \mathop{\longrightarrow}\limits_{\e' \to 0} 0.
    \]
    The above convergence, combined with a union bound, shows that with probability going to $1$ as $\e'\to 0$, for all $n$ the coalescent $((\Part^{n,k}_r(i))_{r \in [k]})_{0\le i \le n}$ does not undergo coalescences due to vertices of degrees $d^n_{i,j}$ between $D^n_i\e_n$ and $D^n_i \e'$.}

    The convergence \eqref{eq cv petites coalescences} comes from Corollary \ref{cor: cv petites coalescences}. The convergence \eqref{eq cv grandes coalescences} holds thanks to Lemma \ref{lem: cv grandes coalescences} and holds jointly with the second one since with high probability coalescences in small degrees do not happen at heights $i$ such that $d^n_{i,1} \ge \e' D^n_i$. Indeed, for all $\eta>0$, by a union bound, the probability of having a small coalescence at a height $\eta n \le i \le (1-\eta) n$ such that $d^n_{i,1} \ge \e' D^n_i$ is upperbounded by
    $$
    \sum_{\eta n \le i\le (1-\eta) n} \1_{d^n_{i,1} > \e' D^n_i} p_{n,i}
    \le 
    \left(\max_{\eta n \le i \le (1-\eta)n} p_{n,i}\right)
    \sum_{\eta n \le i\le (1-\eta) n} \frac{1}{(\e')^2} {\left(\frac{d^n_{i,1}}{{D^n_i}}\right)^2} \1_{d^n_{i,1} > \e' D^n_i}
    \cv[n] 0,
    $$
    where the convergence comes from {\eqref{eq:pas de coalescence avec proba non nulle},} \eqref{eq:cv atomes L2} and \eqref{eq p n i petit}.
    Finally, \eqref{eq cv hauteurs unif} comes from \ref{hyp:naissance} and holds jointly with \eqref{eq cv petites coalescences} and \eqref{eq cv grandes coalescences} since the conditional law of $((\Theta^n_{i,r})_{r \in [k], i\ge 0},(X^n_{q,r,i})_{1 \le q<r \le k, i\ge 0})$ does not depend on $(\h(V^n_r))_{r \in [k]}$.
\end{proof}

\section{GHP convergence} \label{sec:proofGHP}
To show the GHP convergence of $\T^n$, by \cref{B.2}, it suffices to prove that $\T^n$ is close to the genealogies  of a fixed large number of random vertices (that is their paths to the root). We would like to use the same method as  \cite{BBKKbis23+}. {Roughly speaking, this method consists in getting a lower bound for the number of ancestors at a given height of $k$ typical vertices, and then in showing that the probability that another given vertex coalesces with one of these ancestors (on a well chosen time interval) is bounded away from zero.} However, we cannot as we have {a} poor estimate on the number of genealogies that still have not yet merged at {a} given height. 

To avoid this issue, we look instead at the genealogies of vertices chosen such that at each height exactly $k$ of them have not yet merged. We call the union of the genealogies a $k$-trail as those sets have at each height $k$ vertices and consists of paths going to the root (see Figure \ref{fig:3trail}). We construct inductively $k$-trails (\cref{lem:trailconstruct}), and show that $k^2$-trails tend to be close to $k$-trails (\cref{lem:chain}). By a chaining argument this gives us that $\T^n$ is close to $k$-trails (\cref{pro:chain}). Finally, we prove that $k$-trails stay close to the genealogies of a large number of random vertices, which allows us to conclude.

\begin{figure} 
\centering
\includegraphics[scale=0.7]{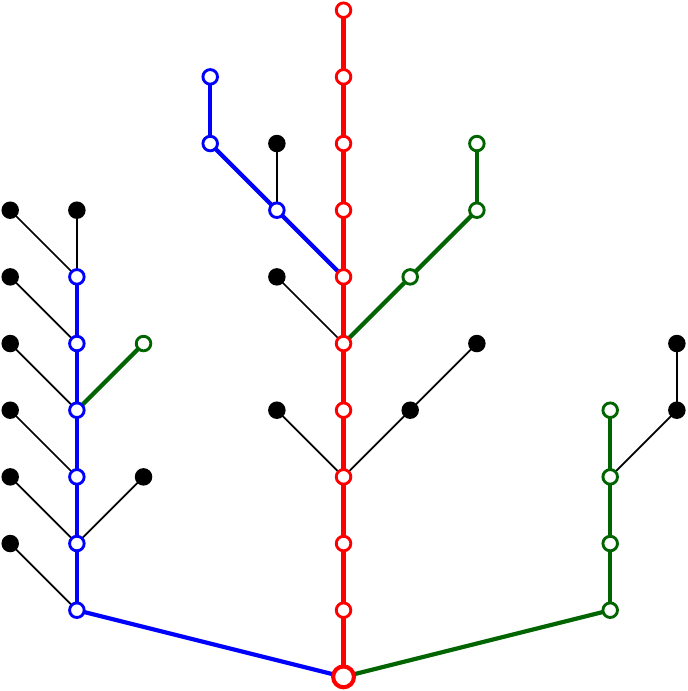}
\caption{An example of 3-trail. The root is at the bottom. The vertices of the 3-trail are represented as empty circle{s}. The vertices of the first trail $\{X_{i,1}\}_{1\leq i \leq \h^n}$ with their edge to their father are represented in red. The second trail $\{X_{i,2}\}_{1\leq i \leq \h^n}$ is in blue. And the third $\{X_{i,3}\}_{1\leq i \leq \h^n}$ is in green. Note that at each height there are exactly 3 vertices in the 3-trail, except when in the tree there are less than $k$ vertices at this height in which case all of them are in the 3-trail. Note also that the fathers of the vertices in the 3-trail are also in the 3-trail.} \label{fig:3trail}
\end{figure}

\subsection{Using $k$-trails to approximate $\T^n$}
For every $n$, $i\in \N$, $j\leq D_{i-1}^n$ we denote by $\fath^n(i,j)=\fath^n((i,j))$ the father of the vertex $(i,j)$ in the tree $\T^n$. Also, for $i\geq 0$ let $\fath^n_i$ be the i-$th$ iterate of $\fath^n$.
In other words, $\fath^n_i(x)$ is the ancestor of $x$ with height $i$ below $x$.

In this whole section, $(X^n_{i,j})_{i\leq \h^n,j\in \N}$ is a family of random variables satisfying the following properties:
\begin{itemize}
\item[\textcolor{link}{\optionaldesc{$(\h)$}{cons:height}}] For every $i\leq \h^n$, for every $j\in \N$, the random variable $X^n_{i,j}$ is a vertex of $\T^n$ of height $i$. 
\item[\textcolor{link}{\optionaldesc{$(\neq)$}{cons:neq}}] For every $i\leq \h^n$, for every $j\neq j'\leq D_{i-1}^n$, we have a.s. $X^n_{i,j}\neq X^n_{i,j'}$.
\item[\textcolor{link}{\optionaldesc{$(\fath)$}{cons:fath}}] For every $1\leq i \leq \h^n$, $j\leq D_{i-1}^n$, a.s. $\fath^n(X^n_{i,j})\in \{X^n_{i-1,j'}\}_{j'\leq j}$.
\item[\textcolor{link}{\optionaldesc{$(\perp)$}{cons:inde}}] For every $1\leq H \leq \h^n$,  the $\sigma$-algebra generated by $ (X^n_{i,j})_{H\leq i\leq \h^n, j\in \N}$ and $(\fath^n(i,j))_{i>H,j\leq D_{i-1}^n})$ is independent from $(\fath^n(i,j))_{1\leq i\leq H,j\leq D_{i-1}^n}$.
\end{itemize}
For $k\in \N$, we call the set $\trail^n_{k}\coloneqq\{X_{i,j}^n\}_{i\leq \h^n, j\leq k'}$ the $k$-trail (associated to $(X_{i,j}^n)_{i,j}$). Note that in $\trail^n_{k}$ there are exactly $\inf(k,D^n_{i-1})$ vertices of height $i$. 

\begin{lemma} \label{lem:trailconstruct} There exist some random variables $(X^n_{i,j})_{i\leq \h^n,j\in \N}$ satisfying \ref{cons:height}, \ref{cons:neq}, \ref{cons:fath}, \ref{cons:inde}.
\end{lemma}
\begin{proof} We construct the sequence by induction on $j$ and downward induction on $i$. For $j=1$ it suffices to take a uniform random vertex at height $\h^n$ and its ancestors. Assume that for every $1\leq i \leq \h^n$ and $j\leq k$, the vertex $X^n_{i,j}$ is defined. Then we let $X_{\h^n,k+1}^n$ be 
\begin{itemize} 
\item If possible, a uniform random vertex of $\T^n$ of height $\h^n$ not in $\{X_{\h^n,j}^n\}_{1\leq j \leq k}$;
\item Else, any vertex of height $\h^n$ in $\T^n$.
\end{itemize}
Then for every $i<\h^n$ given $(X_{h,j}^n)_{h\leq \h^n, 1\leq j\leq k}$ and $X_{i+1,k+1}^n$ we define $X_{i,k+1}^n$ as:
\begin{itemize}
    \item $\fath^n(X_{i+1,k+1}^n)$ if it is not in $\{X_{i,j}^n\}_{1\leq j \leq k}$;
    \item Else, if possible, any vertex in $\T^n$ not in $\{X_{i,j}^n\}_{1\leq j \leq k}$ of height $i$;
    \item Else, if not possible, any vertex in $\T^n$ of height $i$.
\end{itemize}
A quick induction then shows that this sequence is well defined and satisfies the desired properties.
\end{proof}

The aim of this section is to prove that $k$-trails approximate well $\T^n$.

\begin{proposition} \label{pro:chain} For any choice of $(X_{i,j}^n)$ we have
\[ \lim_{\e \to 0} \lim_{k\to \infty} \limsup_{n\to \infty} \proba{d_H(\trail_k^n,\T^n)>\e n}=0. \]
\end{proposition}
First, we prove that $k$-trails and $k^2$-trails are close.  For every $n\in \N$, we write $\|D^n\|_\infty \coloneqq\max_{0\leq i \leq \h^n} D_i^n$ and let $\Gamma_{<\beta n}$ be the set of vertices of $\T^n$ which are at height smaller than $\beta n$.
\begin{lemma}\label{lem:chain} For every $0<\alpha<\beta<1$, for every $n\in \N$, $k\leq \|D^n\|_\infty$ large enough, writing $L=\lceil n/(\log\log k)^2\rceil$, with probability at least $1-1/k$, we have
\[ \forall x\in \trail_{k^2}^n,\quad  \alpha n+9L<\h(x)<\beta n \Longrightarrow d^n(x,\trail_{k}^n \cap \Gamma_{<\beta n})\leq 9 L. \]
\end{lemma}
\begin{proof} First note that if $d^n(X^n_{i,j},\trail^n_k \cap \Gamma_{\beta n})>9L$ for some $\alpha n +9L <i < \beta n$ and $j \in [k^2]$, then for every $1\leq a \leq L$, we have $d^n(\fath^n_a(X^n_{i,j}),\trail^n_k \cap \Gamma_{\beta n})>8L$, so by \ref{cons:fath},
\[ \#\{\alpha n+8L<i'<\beta n,1\leq j'\leq k^2 \text{ such that } d^n(\fath^n(X^n_{i',j'}),\trail^n_k \cap \Gamma_{<\beta n})>8L\}\geq L.\]
Thus, by Markov's inequality it is enough to prove that for every $\alpha n+8L<i<\beta n$, $1\leq j\leq k^2$ we have
\begin{equation} \proba{d^n(X^n_{i,j},\trail^n_k \cap \Gamma_{<\beta n})>8L}\leq \frac{1}{k} \frac{L}{nk^2}
\label{eq:half_lem:chain}.\end{equation}
Fix such $i,j$. To this end, we can first compute explicitly for every $1\leq a\leq 4L$, 
\begin{equation} \proba{\fath^n_{a+1}(X^n_{i,j}) \notin \{X^n_{i-a-1,j}\}_{1\leq j \leq k} |\fath^n_{a}(X^n_{i,j}) \notin \{X^n_{i-a,j}\}_{1\leq j \leq k}}. \label{eq:partialprobakfuse} \end{equation}
Indeed, by \ref{cons:inde}, given the event $\fath^n_{a}(X^n_{i,j}) \notin \{X^n_{i-a,j}\}_{1\leq j \leq k}$, the fathers of $\fath^n_{a}(X^n_{i,j})$ and of the $X^n_{i-a,j}$'s for ${1\leq j \leq k}$ have the same law as the (unconditioned) father{s} of any $k+1$ different vertices of height exactly $i-a$. In other words, writing $i'=i-a-1$, if $D_{i'}^n>k$ (note that by \ref{cons:neq} when $D^n_{i-1-a}\leq k$, a.s. $\fath^n_{a+1}(X^n_{i,j})\in \{X^n_{i-a-1,j}\}_{1\leq j \leq k}$) then \eqref{eq:partialprobakfuse} equals,  dividing according to $({i',\ell})$ the value of $\fath^n_{a+1}(X^n_{i,j})$,
\[ \sum_{\ell\ge 1} \frac{d_{i',\ell}^n}{D_{i'}^n} \prod_{1\leq b\leq k}\left (1-\frac{d_{i',\ell}^n-1}{D_{i'}^n-b}\right ),\]
which we may upper-bound, using that $b\geq 1$ and then that $(1-x)^k\leq 1-\min(1,kx)/2$ for all $x \in [0,1]$, by
\[ \sum_{\ell\ge 1} \frac{d_{i',\ell}^n}{D_{i'}^n} \left(1-\frac{d_{i',\ell}^n-1}{D_{i'}^n-1}\right )^k \leq \sum_{\ell\ge 1} \frac{d_{i',\ell}^n}{D_{i'}^n}\left (1-\min\left (1,k\frac{d_{i',\ell}^n-1}{D_{i'}^n-1} \right )/2\right ).\]
The right-hand side is then equal to $1-\tau_{i'}^n(k)/2$ since $\sum_{\ell\ge1} d_{i',\ell}^n=D_{i'}^n$ and by definition of $\tau$ (see \eqref{eq def tau} in the introduction), 
which we may again upper-bound by $e^{-\tau_{i'}^n(k)/2}$ since for every $x>0$, we have $(1-x)\leq e^{-x}$. 
Thus, 
\begin{align*}  \proba{d^n(X^n_{i,j},\trail^n_k\cap \Gamma_{<\beta n})>8L} & \leq \proba{\fath^n_{8L}(X^n_{i,j})\notin \{X^n_{i-8L,j}\}_{1\leq j \leq k}}
\\ & \le \prod_{a=0}^{8L-1} \proba{\fath^n_{a+1}(X^n_{i,j}) \notin \{X^n_{i-a-1,j}\}_{1\leq j \leq k} |\fath^n_{a}(X^n_{i,j}) \notin \{X^n_{i-a,j}\}_{1\leq j \leq k}}
\\ & \leq \prod_{i-8L<h\leq i} e^{-\tau_{h}^n(k)/2} 
\\ & =\exp \left( -\tau_{i-8L+1,i}^n(k)/2 \right )
\\ & \leq k^{-7/2},
\end{align*}
using for the last inequality $L=\lceil n/(\log\log k)^2\rceil$ and \ref{hyp:tightGHP}. This implies the desired bound \eqref{eq:half_lem:chain}.
\end{proof}
\begin{proof}[Proof of Proposition \ref{pro:chain}] We deduce Proposition \ref{pro:chain} from Lemma \ref{lem:chain} with a  chaining argument. Fix $0<\e<0.01$. Fix $K$ large enough depending only on $\e$ such that $\log \log K>2/\e$, and such that for every $n\in \N$ large enough, Lemma 4.3 holds with $\alpha=\e/2$, $\beta=1-\e$, and every $K\leq k \leq \|D^n\|_\infty$. Let $N=N(n,K)$ be the smallest integer such that $K^{2^N}\geq \|D^n\|_\infty$. { Note already that by putting $k=\|D^n\|_\infty$ in \ref{hyp:tightGHP} we get 
\[ \log(\|D^n\|_\infty) \leq \tau^n_{i,i+n/(\log \log \|D^n\|_\infty)} \leq 1+n/(\log \log\|D^n\|_\infty)=o(n)\]
as $n \to \infty$, so $N=o(n)$.}

By Lemma \ref{lem:chain}, for every $n$ large enough with probability at least $1-\sum_{i=0}^{N-1} 1/K^{2^i}\geq 1-2/K$:
\begin{equation} \forall 0\leq i \leq N-1, \quad \forall x\in \trail_{K^{2^{i+1}}}^n,\quad  \e n< \h(x)<(1-\e) n \Longrightarrow d^n(x,\trail_{K^{2^i}}^n \cap \Gamma_{<(1-\e) n})\le 9 \lceil n/(\log \log K^{2^i})^2 \rceil . \label{eq:middlechain} \end{equation}   
Then writing $\Gamma_{\leq \e n}$ for the set of vertices with height at most $\e n$,  \eqref{eq:middlechain} implies 
\begin{equation*} \forall 0\leq i \leq N-1, \qquad d_H\left ((\Gamma_{\leq \e n}\cup \trail_{K^{2^i}}^n)\cap \Gamma_{<(1-\e) n} \,, \,  (\Gamma_{\leq \e n}\cup \trail_{K^{2^{i+1}}}^n)\cap \Gamma_{<(1-\e) n} \right )\leq 9 \lceil n/(\log \log K^{2^i})^2 \rceil, \end{equation*} 
which by the triangle inequality implies for every $n,K$ large enough, 
\begin{align*} d_H \left ((\Gamma_{\leq \e n}\cup \trail_K^n)\cap \Gamma_{<(1-\e) n}\, , \, (\Gamma_{\leq \e n}\cup \trail_{K^{2^N}}^n)\cap \Gamma_{<(1-\e) n} \right ) 
&\leq \sum_{i=0}^{N-1} 9 \lceil n/(\log \log K^{2^i})^2 \rceil \\
&\leq N+9n \sum_{i=1}^\infty \frac{1}{(i\log(2)+\log \log K)^2}\\
&\le N+\frac{40 n}{\log \log K}.
\end{align*}
Then note that by \ref{cons:neq} and since $K^{2^N}\geq \|D^n\|_\infty$, we know that $\trail_{K^{2^N}}^n$ contains all vertices of $\T^n$. Moreover, $(0,1)\in \trail_K^n$ so every $x\in \Gamma_{\leq \e n}$ is at distance at most $\e n$ of $\trail_K^n$. Finally, if $x$ is a vertex of $\T^n$ such that $\h (x) \ge (1-\e)n$ then by considering the highest ancestor $x'$ of $x$ belonging to $\Gamma_{<(1-\e)n}$, we obtain that $d(x,\Gamma_{\le \e n} \cup \trail^n_K ) \le 1+ n\e +d(x',\Gamma_{\le \e n} \cup\trail^n_K)$.  Hence, the last line implies 
\[ d_H \left (\trail_K^n, \T^n \right ) \leq 1+ 2\e n + N+\frac{40n}{\log \log K}. \]
Since $\e$ can be chosen arbitrarily small, $K$ arbitrarily large, and since $N=o(n)$ this concludes the proof.
\end{proof}
\subsection{Approximating $k$-trails from genealogies of random vertices}
By Proposition \ref{B.2}, to extend the GP convergence in Theorem \ref{th:GP} to the GHP convergence in Theorem \ref{th:GHP} it suffices to prove the following strong leaf tightness criterion:
\[ \forall \delta>0,\qquad  \lim_{k\to \infty} \limsup_{n\to \infty} \proba{d_H(\T^n,\{V^n_1,V^n_2,\ldots,V^n_k\})>\delta n } = 0. \]
where $V^n_1,V^n_2,\dots$ are i.i.d. uniform random vertices in $\T^n$. To that end, by Proposition \ref{pro:chain}, it is enough to prove that 
\[ \forall \delta>0, \forall K\in \N, \qquad \lim_{k\to \infty} \limsup_{n\to \infty} \proba{\max_{x\in \trail_K^n} d^n(x,\{V^n_1,V^n_2,\ldots,V^n_k\})>\delta n} = 0. \]
Or even, by noting that if there exists an $x\in \trail_K^n$ such that $d^n(x,\{V^n_1,V^n_2,\ldots,V^n_k\})>\delta n$ then either:
\begin{itemize}
\item $\h(x)<(\delta/2)n$ and $\inf_{1\leq i \leq k} \h(V^n_i)\ge (\delta/2) n$. For every $\delta>0$, the probability of this event goes to $0$ as $k\to \infty,n\to \infty$ by \ref{hyp:naissance} and since $\nu$ has support $[0,1]$.
\item Or $\h(x)\ge(\delta/2)n$, so that by \ref{cons:fath} there are at least $\lfloor (\delta/4) n\rfloor$ different vertices $y\in \trail_k^n$ such that  $d^n(y,\{V^n_1,V^n_2,\ldots,V^n_k\})\ge \lfloor (\delta/4) n \rfloor$.
\end{itemize}
{As a result, 
\begin{align*}
    \limsup_{k\to \infty} \limsup_{n\to \infty} \ &\proba{\max_{x\in \trail_K^n} d^n(x,\{V^n_1,V^n_2,\ldots,V^n_k\})>\delta n}\\ &\le
    \limsup_{k\to \infty} \limsup_{n\to \infty} \proba{ \# \{y \in \trail_K^n, \ d^n(y,\{V^n_1,V^n_2,\ldots,V^n_k\})\ge\lfloor (\delta/4)n \rfloor \} \ge \lfloor (\delta/4)n \rfloor} \\
    &\le \limsup_{k\to \infty} \limsup_{n\to \infty} \frac{1}{\lfloor (\delta/4)n \rfloor} \E\left[ \# \{y \in \trail_K^n, \ d^n(y,\{V^n_1,V^n_2,\ldots,V^n_k\})\ge\lfloor (\delta/4)n \rfloor \} \right] \\
    &\le \limsup_{k\to \infty} \limsup_{n\to \infty} \frac{K \h^n}{\lfloor (\delta/4)n \rfloor} \max_{1\leq a\leq \h^n,1\leq b\leq K} \proba{d^n(X^n_{a,b},\{V^n_1,V^n_2,\ldots,V^n_k\})\ge \lfloor (\delta/4) n\rfloor},
\end{align*}
where in the third line we applied Markov's inequality and in the last line we used the fact that $\#\trail_K^n \leq K \h^n$. But since $\h^n \sim n$,} it is enough to show that 
\[ \forall \delta>0, \forall K>0, \qquad \lim_{k\to \infty} \limsup_{n\to \infty} \max_{1\leq a\leq \h^n,1\leq b\leq K} \proba{d^n(X^n_{a,b},\{V^n_1,V^n_2,\ldots,V^n_k\})>{5}\delta n} = 0, \]
or even by considering the different possible values of $X^n_{a,b}$ it suffices to prove the following result:
\begin{proposition} \label{pro:MiniTightGHP}Under the setting of Theorem \ref{th:GHP},  for every $\delta>0$, 
\[ \lim_{k\to \infty} \limsup_{n\to \infty} \max_{1\leq a\leq \h^n,1\leq b\leq D^n_{a-1}}  \proba{d^n((a,b),\{V^n_1,V^n_2,\ldots,V^n_k\})>{5}\delta n} = 0. \]
\end{proposition}
\begin{proof}
We follow directly the same approach as in Section \ref{sec:continuum_leaf-tight} which proves an analogous result for the limit tree. Again by \ref{hyp:naissance} and since $\supp(\nu)=[0,1]$ we may restrict {ourselves} to the case where $a>2\delta n$. And as in Section \ref{sec:continuum_leaf-tight}, since we may consider a larger but still bounded number of random vertices $V_1^n, V_2^n,\dots, V_k^n$ and since $\nu$ has support $[0,1]$, we can assume them to have height between $a-\delta n$ and $a$. In other words, it suffices to check that for every $\delta>0$, 
\[ \lim_{k\to \infty} \limsup_{n\to \infty} \max_{2\delta n< a\leq \h^n,1\leq b\leq D^n_{a-1}}  \proba{d^n((a,b),\{V^n_1,V^n_2,\ldots,V^n_k\})> {5}\delta n\middle | \forall 1\leq i \leq k, a-\delta n<\h(V^n_i)<a } = 0. \]

To ease the writing, with a slight abuse of notation, let $\probak{\cdot}\coloneqq\proba{\cdot \middle | \forall 1\leq i \leq k, a-\delta n<\h(V^n_i)<a}$ (define similarly $\E^k$) and let $V_0^n=(a,b)$. Again we may consider the number of vertices that have not coalesced up to time $x$ that is: 
\[ N_k(x)\coloneqq\#\{1\leq i \leq k \text{ such that }  \h(V_i)\geq x\text{ and } \forall 1\leq j \neq i \leq k, \ \h(V^n_i\wedge V^n_j)<x\}, \]
where $V^n_i \wedge V^n_j$ is the nearest common ancestor of $V^n_i$ and $V^n_j$ in $\T^n$.
And, for exactly the same reasons as in Section \ref{sec:continuum_leaf-tight}, one can see that 
$$
\P^k\left(d^n(V^n_0, \{V^n_1, V^n_2, \ldots, V^n_k\}) > {5}\delta n\right) \le \frac{\E^k\left[N_k(a-2\delta n )\right]}{k}.
$$
The proof of Proposition \ref{pro:MiniTightGHP} is thus complete as long as we prove the following result.
\end{proof}
\begin{lemma} For every $\delta>0$, 
\[ \lim_{k \to \infty} \limsup_{n \to \infty} \max_{2\delta n < a \le \h^n}\frac{\E^k[N_k(a-2\delta n)]}{k}= 0. \]
\end{lemma}
\begin{proof} The proof is essentially the same as for Lemma \ref{lem:end_lim_leaf-tight} with slightly different computations. Let $\delta>0$. Let $k\ge 1$, $n\ge 1$ and $2\delta n < a \le \h^n$. Recall the definition of the sequence $(\e_n)$ from Subsection \ref{sous-section coalescent discret}. For all $a- 2\delta n \le i \le a - \delta n$, we lowerbound
\begin{align*}
\E^k &\left[ N_k(i+1) - N_k(i) \vert N_k(i+1)\right]\\
&\ge 
\sum_{j=1}^\infty \1_{d^n_{i,j}>\e_n D^n_i} N_k(i+1)\frac{d^n_{i,j}}{D^n_i} \left(1 - \prod_{r=1}^{N_k(i+1)-1} \left(1- \frac{d^n_{i,j}-1}{D^n_i-r}\right)\right)
+ 2\binom{N_k(i+1)}{2} \sum_{j=1}^\infty \1_{d^n_{i,j} \le \e_n D^n_i} \frac{d^n_{i,j} (d^n_{i,j} -1)}{D^n_i(D^n_i-1)}\\
&\ge 
\sum_{j=1}^\infty \1_{d^n_{i,j}>\e_n D^n_i} N_k(i+1)\frac{d^n_{i,j}}{D^n_i} \left(1 -  \left(1- \frac{d^n_{i,j}-1}{D^n_i}\right)^{N_k(i+1)-1}\right)
+ 2\binom{N_k(i+1)}{2} \sum_{j=1}^\infty \1_{d^n_{i,j} \le \e_n D^n_i} \frac{d^n_{i,j} (d^n_{i,j} -1)}{D^n_i(D^n_i-1)}\\
&\ge \sum_{j=1}^\infty \1_{d^n_{i,j}>\e_n D^n_i \vee (1+ D^n_i/(N_k(i+1)-1))} N_k(i+1) \frac{d^n_{i,j}}{D^n_i} \left(1-\frac{1}{e}\right) + 2\binom{N_k(i+1)}{2} \sum_{j=1}^\infty \1_{d^n_{i,j} \le \e_n D^n_i} \frac{d^n_{i,j} (d^n_{i,j} -1)}{D^n_i(D^n_i-1)}
,
\end{align*}
where in the last inequality we used that for all $x\in (0,1),y>0$ such that $xy>1$ we have $(1-x)^y\leq 1/e$.

Next, fix $\eta>0$. Considering only the case where $N_k(i+1)\ge \eta k$ and taking the expectation we get 
\begin{align*}
\E^k &\left[ N_k(i+1) - N_k(i) \right]\\
&\ge \sum_{j=1}^\infty \1_{d^n_{i,j}>\e_n D^n_i \vee (1+ D^n_i/(\eta k-1))} \E^k[N_k(i+1)] \frac{d^n_{i,j}}{D^n_i} \left(1-\frac{1}{e}\right) + \E^k\left[2\binom{N_k(i+1)}{2}\right] \sum_{j=1}^\infty \1_{d^n_{i,j} \le \e_n D^n_i} \frac{d^n_{i,j} (d^n_{i,j} -1)}{D^n_i(D^n_i-1)}.
\end{align*}
We then sum over $a-2\delta n \le i \le a - \delta n$ and use the fact that $N_k \le k$, which gives
\begin{align*}
k\ge \sum_{a-2 \delta n \le i \le a-\delta n}&\sum_{j=1}^\infty \1_{d^n_{i,j}>\e_n D^n_i \vee (1+ D^n_i/(\eta k-1))} \E^k[N_k(i+1)] \frac{d^n_{i,j}}{D^n_i} \left(1-\frac{1}{e}\right) \\
+ &\sum_{a-2 \delta n \le i \le a-\delta n}\E^k\left[2\binom{N_k(i+1)}{2}\right] \sum_{j=1}^\infty \1_{d^n_{i,j} \le \e_n D^n_i} \frac{d^n_{i,j} (d^n_{i,j} -1)}{D^n_i(D^n_i-1)}.
\end{align*}
But since for all $a-2 \delta n \le i \le a- \delta n$ we have $N_k(i) \ge N_k(a-2\delta n)$, using Markov's inequality we obtain that
\begin{align}
k\ge &\P^k(N_k(a- 2 \delta n)\ge \eta k) (\eta k-1)\notag\\
&\times\sum_{a-2 \delta n \le i \le a-\delta n} \left(\sum_{j=1}^\infty \1_{d^n_{i,j}>\e_n D^n_i \vee (1+ D^n_i/(\eta k-1))}  \frac{d^n_{i,j}}{D^n_i} \left(1-\frac{1}{e}\right) 
+  (\eta k-1)\sum_{j=1}^\infty \1_{d^n_{i,j} \le \e_n D^n_i} \frac{d^n_{i,j} (d^n_{i,j} -1)}{D^n_i(D^n_i-1)} \right).\label{eq majoration nombre de gens qui ne coalescent pas}
\end{align}
Finally, by \eqref{eq:cv rho epsilon n}, by \ref{hyp:atomes} and $\h^n\sim n$, we deduce that
\begin{align*}
    &\liminf_{n\to \infty} \min_{2 \delta n < a \le \h^n}
    \sum_{a-2 \delta n \le i \le a-\delta n} \left(\sum_{j=1}^\infty \1_{d^n_{i,j}>\e_n D^n_i \vee (1+ D^n_i/(\eta k-1))}  \frac{d^n_{i,j}}{D^n_i} \left(1-\frac{1}{e}\right) 
+  (\eta k-1)\sum_{j=1}^\infty \1_{d^n_{i,j} \le \e_n D^n_i} \frac{d^n_{i,j} (d^n_{i,j} -1)}{D^n_i(D^n_i-1)} \right)\\
&\ge \min_{1 \le \ell \le \lceil 2/\delta \rceil}
\left( \left(1-\frac{1}{e} \right) \sum_{(\ell-1) \delta/2 < t < \ell \delta/2} \sum_{j=1}^\infty \1_{\theta_j(t)>1/(\eta k-1)} \theta_j(t) + (\eta k-1) \rho([(\ell-1)\delta/2, \ell \delta/2])
\right)\\
&\cv[k] \infty,
\end{align*}
where the last convergence stems from \ref{hyp:tightGP}. Taking \eqref{eq majoration nombre de gens qui ne coalescent pas} into account, we have thus proven that for all $\eta>0$,
$$
\lim_{k\to \infty} \limsup_{n \to \infty} \max_{2 \delta n <a < \h^n} \P^k(N_k(a-2\delta n) \ge \eta k) = 0.
$$
This concludes the proof since $N_k \le k$.
\end{proof}
\section{Application to BGW trees in varying environment} \label{sec:appBGWTVE}

Let us describe here how our result gives general conditions under which Bienaymé--Galton--Watson trees in varying environment stopped at height $n$ satisfy a scaling limit for the GP and the GHP topology. We rely on the scaling limit of the profile obtained by Bansaye and Simatos in \cite{BS15}. This section can be read just after the introduction. We refer to \cite{BS15} for examples. {Let us stress that our conditions in this section are less general than in the previous sections, inasmuch as here we make an assumption which rules out the case where two high degrees are allowed to be at the same height (see Remark \ref{remarque pas de grands degrés à la meme hauteur}). This assumption simplifies the analysis since in this section, the jumps of the scaling limit of the profile will exactly correspond to the scaling limit of the high degrees.}

As in \cite{BS15}, for all $n \ge 1$, independently for all $1 \le i \le n$, let $(\xi_{i,n}(j))_{j\ge 1}$ be a family of i.i.d.\@ random variables in the set of non-negative integers. We recall that the Bienaymé--Galton--Watson process in varying environment $(Z_{i,n})_{0 \le i \le n}$ is defined by setting $Z_{0,n}=1$ and for all $i \in [n]$, 
$$
Z_{i,n} = \sum_{j=1}^{Z_{i-1,n}} \xi_{i,n}(j).
$$
Note that the environment, i.e.\@ the law of the $\xi_{i,n}(j)$'s, depends on $i$. The Bienaymé--Galton--Watson tree in varying environment $\T_n$ of height $n$ may be constructed conditionally on $(\xi_{i,n}(j))_{j\ge 1, 0 \le i \le n}$ as the tree with vertices of fixed degrees and heights obtained by letting for all $0 \le i\le n-1$ the sequence $(d_{i,j}^n)_{1 \le j \le Z_{i,n}}$ be a reordering of the $\xi_{i+1,n}(j)$'s for $1 \le j \le Z_{i,n}$ in the non-increasing order, by setting $d^n_{i,j}=0$ {for all $i,j$ such that} $j > Z_{i,n}$ or $i \ge n$.

We mainly follow the notation of \cite{BS15} except that here $n$ is the scale of the height (i.e.\@ the time scale for the GWVE process). Let $(\ell_n)_{n\ge 1}$ be a sequence of positive real numbers tending to $\infty$ which will be the scale of the profile (i.e.\@ the space scale). 
To simplify the notation, we write $\xi_{i,n} = \xi_{i,n}(1)$ for all $i \in [n]$. As in \cite{BS15}, we define the rescaled quantities for all integers $n\ge 1,i \in [n]$ and for all $t \in [0,1]$,
$$
\overline{\xi}_{i,n} = \frac{1}{\ell_n} (\xi_{i,n} -1),
\qquad \alpha_{i,n} = \E\left[\frac{\overline{\xi}_{i,n}}{1+\overline{\xi}_{i,n}^2}\right], \qquad \beta_{i,n} = \E\left[\frac{\overline{\xi}_{i,n}^2}{1+\overline{\xi}_{i,n}^2} 
\right] \qquad \text{and} \qquad 
X_n(t)= \frac{1}{\ell_n} Z_{\lfloor nt \rfloor,n}.
$$
Intuitively, $\alpha_{i,n}$ is a proxy for the drift and $\beta_{i,n}$ plays the role of a ``variance'' of $\overline{\xi}_{i,n}$. We then set for all $n\ge 1, t\in [0,1]$, $x>0$,
$$
\alpha_n(t) = \ell_n \sum_{i={1}}^{\lfloor nt \rfloor} \alpha_{i,n},
\quad
\beta_n(t) = \frac{1}{2} \ell_n \sum_{i={1}}^{\lfloor nt \rfloor }\beta_{i,n}  \quad
\text{and} \quad
\mu_n([x,\infty)\times(0,t]) = \ell_n \sum_{i={1}}^{\lfloor nt \rfloor } \P(\overline{\xi}_{i,n} \ge x).
$$
 Next, {we use a simpler version of} Assumption A of \cite{BS15} in our framework. We assume that there exist a càdlàg 
function of finite variation $\alpha: [0,1]\to \R$, an increasing càdlàg
function $\beta:[0,1] \to \R$ and a positive measure $\mu$ on $(0,\infty)\times (0,1)$ such that
\begin{description}[topsep=0pt,itemsep=-1ex,partopsep=1ex,parsep=1ex,labelwidth=1.6cm,leftmargin=!,align=CenterWithParen]
\item[\textcolor{link}{\optionaldesc{$(\mathsf{A1})$}{hyp:A1}}] For all $t \in [0,1],x> 0$ such that $\mu(\{x\} \times (0,t])=0$,
\begin{align*}
\alpha_n(t)\cv[n] \alpha(t), \qquad &\lVert \alpha_n \rVert (t) \cv[n] \lVert \alpha \rVert (t), \qquad \beta_n(t) \cv[n] \beta(t)\\
\enskip &\text{and} \enskip
\mu_n([x,\infty)\times(0,t]) \cv[n] \mu([x,\infty)\times(0,t]),
\end{align*}
where $\lVert \cdot \rVert$ is the total variation.
\item[\textcolor{link}{\optionaldesc{$(\mathsf{A'2})$}{hyp:A2}}] {The functions $\alpha$, $\beta$ are continuous and the measure $\mu((0,\infty) \times \mathrm{d} t)$ has no atoms.}
\end{description}
{Note that \ref{hyp:A1} implies the weak convergence of the restriction of $\mu_n$ to $[x, \infty) \times (0,t]$ towards the restriction of $\mu$ to $[x, \infty) \times (0,t]$. Our assumption \ref{hyp:A2} makes Assumption (A2) of \cite{BS15} void since the functions $\alpha, \beta$ and $t\mapsto \mu([x,\infty)\times (0,t])$ are continuous for all $x>0$ by \ref{hyp:A2}.}

We further assume the following condition coming from Proposition 2.2 of \cite{BS15}:
\begin{description}[topsep=0pt,itemsep=-1ex,partopsep=1ex,parsep=1ex,labelwidth=1.6cm,leftmargin=!,align=CenterWithParen]
\item[\textcolor{link}{\optionaldesc{$(\mathsf{A3})$}{hyp:nobottleneck}}] For every $C>0$, 
$$
\liminf_{n\to \infty} \left( \inf_{1 \le i \le n }
\E\left[\xi_{i,n} \1_{\xi_{i,n \le C \ell_n}}\right]\right)>0.
$$
\end{description}
Then, by Theorem 2.1 of \cite{BS15}, we know that there exists a càdlàg process $(X(t))_{t\in [0,1]}$ taking its values in $[0,\infty]$ such that
\begin{equation}\label{eq cv GWVE}
(X_n(t))_{t \in [0,1]}\cvloi[n](X(t))_{t\in [0,1]}
\end{equation}
for the Skorokhod J1 topology on the space $\mathbb{D}([0,1], [0,\infty])$, where $[0,\infty]$ is equipped with the distance $d(x,y) = \vert e^{-x}-e^{-y}\vert$, and $\infty$ is an absorbing state of the process $(X(t))_{t\in [0,1]}$. As in the first point of Theorem 2.1 of \cite{BS15}, we define the {(continuous)} function $\widetilde{\beta}$ by setting for all $t \in [0,1]$,
$$
\widetilde{\beta}(t) = \beta(t) - \frac{1}{2} \int_{(0,\infty) \times (0,t]} \frac{x^2}{1+x^2} \mu(\mathrm{d}x,\mathrm{d}t).
$$

On the event ``$X(1) \in (0,\infty)$'', we define the measures on $(0,1)$
$$
\nu(\mathrm{d}t) = \frac{X(t)\mathrm{d}t}{\int_0^1 X(s) \mathrm{d}s} \qquad \text{and} \qquad
\rho(\mathrm{d}t)  = \frac{2\widetilde{\beta}(\mathrm{d}t)}{X(t)},
$$
where $\widetilde{\beta}(\mathrm{d}t)$ is the Lebesgue-Stieltjes measure associated with $\widetilde{\beta}$. 
We also let $(\Delta_+X(t))_{t \in (0,1)}$ be the collection of the positive jumps of $X$ in $(0,1)$ and for all $t\in (0,1)$, we set $\theta_1(t) = \Delta_+X(t)/X(t)$ and $\theta_j(t)=0$ for all $j\ge 2$. We set $\Theta = ((\theta_j(t))_{j\ge 1})_{t\in (0,1)}$.

By Skorokhod's representation theorem, we may assume that \eqref{eq cv GWVE} holds almost surely. Then the consequence of Theorems \ref{th:GP} and \ref{th:GHP} is the following:
\begin{corollary}\label{corollaire CWVE}
    Suppose that $\alpha, \beta, \mu$ satisfy the conditions \ref{hyp:A1}, \ref{hyp:A2} and \ref{hyp:nobottleneck}. 
    Assume that $\P(\forall t \in (0,1], \ 0<X(t)<\infty)>0$.
    \begin{itemize}
        \item If for all $0<a<b<1$, 
        we have {$\int_{(0,\infty) \times [a,b]} x\mu( \mathrm{d}x, \mathrm{d} t)=\infty$} or 
        $\widetilde{\beta}(b)-\widetilde{\beta}(a)>0$, then on the event ``$\forall t \in (0,1], \ X(t) \in (0,\infty)$'', the tree $\T_n/n$ converges in distribution to the random metric space $\T(\nu, \rho, \Theta)$ for the GP topology.
        \item If for all $\delta>0$, for all $0<a<b<1$, there exist $n_0,k_0 \ge 0$ such that for all $n\ge n_0$, on the event ``$\forall t \in (0,1], \ X(t) \in (0,\infty)$'', with probability at least $1-\delta$, for all $k_0 \le k \le \max_{1 \le i \le n} Z_{i,n}$, for all $a n \le h \le b n$, 
        \begin{equation}\label{eq hypothese GWVE tight GHP}
        \sum_{i=h}^{h+ n/ (\log \log k)^2} \sum_{j=1}^{Z_{i-1,n}} \frac{\xi_{i,n}(j)}{Z_{i,n}} \min \left( k \frac{\xi_{i,n}(j)-1}{Z_{i,n}-1}, 1\right)\ge \log k,
        \end{equation}
        then the same convergence holds for the GHP topology.
    \end{itemize}
\end{corollary}

{
Let us comment about the assumptions. Note that, as a consequence of \ref{hyp:A1} and \ref{hyp:A2}, by Dini's theorem, for all $x>0$ which is not an atom of $\mu(\mathrm{d} x \times (0,1))$ and for all non-negative continuous function $\varphi$, the convergences of the non-decreasing functions $t \mapsto \lVert \alpha_n \rVert(t), t\mapsto \beta_n(t)$ and $t\mapsto \int_{[x,\infty) \times (0,t]} \varphi(y) \mu_n(\mathrm{d} y, \mathrm{d} s)$ hold uniformly on $[0,1]$ since the limits are continuous.

Before going further, we introduce some more notation. For all $i \in [n]$ and $\e>0$, let $\alpha_{i,n, \e} \coloneqq \E\left[ \1_{\overline{\xi}_{i,n} \le \e} \overline{\xi}_{i,n}/(1+ \overline{\xi}_{i,n}^2) \right]$. Note that \ref{hyp:A1} and \ref{hyp:A2} imply that
\begin{equation}\label{eq alpha i n epsilon tend vers zero}
\lim_{n\to \infty} \ell_n  \sup_{i \in [n]} \left\vert \alpha_{i,n,\e} \right\vert =0
\end{equation}
since the convergences of $t\mapsto \lVert \alpha_n \rVert(t)$ and $t\mapsto \int_{[\e,\infty) \times (0,t]} {x}/{(1+x^2)}\mu_n (\mathrm{d}x, \mathrm{d}s)$ hold uniformly on $[0,1]$.
}
\begin{remark}\label{remarque pas de grands degrés à la meme hauteur}
Note that the assumption that $\mu((0,\infty)\times \mathrm{d}t)$ has no atoms rules out the case where two high degree vertices are at the same height. If one removes this assumption, then the $\theta_j(t)$'s could not be described using only the process $X$. {Similarly, the continuity of $\alpha$ ensures, through \eqref{eq alpha i n epsilon tend vers zero}, that the small degree vertices do not contribute to the jumps of $X$ so that the $\theta_j(t)$'s indeed correspond to the positive jumps of $X$.} {We leave the question of obtaining the scaling limit in the case where $\mu((0,\infty)\times \mathrm{d}t)$ has atoms or when $\alpha$ is not continuous as an open problem.}
\end{remark}
{
\begin{remark}
    One can present the above assumptions in the equivalent form. One can replace \ref{hyp:A1} and \ref{hyp:A2} by the assumptions that \ref{hyp:A1} holds, that $\widetilde{\beta}$ is continuous, that $\mu((0,\infty)\times \mathrm{d} t)$ has no atoms, that Assumption (A2) of \cite{BS15} holds, i.e.\@ for every $t \in [0,1]$ such that $\Delta \alpha(t) \neq 0$ or $\Delta\beta(t) \neq 0$, 
$$
{\ell_n} \alpha_{\lfloor n t \rfloor, n} \cv[n] \Delta \alpha (t)
\qquad \text{and} \qquad {\ell_n} \beta_{\lfloor n t \rfloor, n}\cv[n] \Delta \beta (t) \enskip
$$
and that 
    \[
    \lim_{\e \to 0} \limsup_{n\to \infty} \ell_n  \sup_{i \in [n]} \left\vert \alpha_{i,n,\e} \right\vert =0.
    \]
    The above convergence can be seen as a ``near-criticality'' condition. Clearly, the above assumptions are implied by \ref{hyp:A1} and \ref{hyp:A2}. But they are actually equivalent. Indeed, from the definition of $\widetilde{\beta}$ and the fact that $\mu((0,\infty)\times \mathrm{d} t)$ has no atoms, one can see that $\beta$ is continuous. Moreover, the above two displays and the fact that $\mu((0,\infty)\times \mathrm{d} t)$ has no atoms, together with \ref{hyp:A1} imply the continuity of $\alpha$. One could also add a time-change $\gamma_n : [0,1 ]\to [0,1]$ converging uniformly to $x\mapsto x$ in the definitions of $\alpha_n, \beta_n$ and $\mu_n$ and in the above assumption on jumps: the assumption \ref{hyp:A2} would still be satisfied.
\end{remark}
}
{
\begin{remark}
    The assumption \eqref{eq hypothese GWVE tight GHP} corresponds to the condition \ref{hyp:tightGHP} and involves the behavior of the process at every scale, not only the one of the scaling limit, that is why we believe it has to be checked separately in the presence of large degrees.
\end{remark}
}

{
Before starting the proof, let us introduce the following martingale with respect to the natural filtration $(\mathcal{F}^n_i)_{0 \le i \le n}$ associated to the process $((\xi_{i,n}(j))_{j\ge 1} )_{i\in [n]}$. For all $\e>0$ and $\delta>0$, for all $0 \le k\le n$, we define
\begin{equation}\label{eq martingale}
M^{n,\e, \delta}_k \coloneqq \sum_{i=1}^k \sum_{j=1}^{Z_{i-1,n}} \left( \frac{\overline{\xi}_{i,n}(j)}{1+ \overline{\xi}_{i,n}(j)^2} - \frac{{\alpha}_{i,n,\e}}{\P(\overline{\xi}_{i,n} \le \e)} \right) \1_{\overline{\xi}_{i,n}(j) \le \e \text{ and } \delta \ell_n \le Z_{i-1,n}\le \ell_n/\delta}.
\end{equation}
\begin{lemma}\label{lemme martingale}
For all $n\ge 1$ and $\e,\delta>0$, the process $(M^{n,\e,\delta}_k)_{0 \le k \le n}$ is a martingale with respect to the filtration $(\mathcal{F}^n_k)_{0\le k \le n}$. Moreover, for all $0\le s<t\le 1$,
\[
\lim_{\e \to 0}\limsup_{n\to \infty}\E\left[\sup_{ \lfloor ns \rfloor \le k \le  \lfloor nt\rfloor} \left( M^{n,\e,\delta}_k - M^{n,\e,\delta}_{\lfloor ns \rfloor} \right)^2\right]\le \frac{8(\widetilde{\beta}(t)-\widetilde{\beta}(s))}{\delta}.
\]
In particular, for all $\eta>0$ and $\e'\in (0,1)$,
\[
\lim_{\e \to 0} \limsup_{n\to \infty} \P \left( \sup_{\e'n \le i \le (1-\e')n} \left\vert \sum_{j=1}^{Z_{i-1,n}} \left(\frac{\overline{\xi}_{i,n}(j)}{1+ \overline{\xi}_{i,n}(j)^2} - \frac{{\alpha}_{i,n,\e}}{\P(\overline{\xi}_{i,n} \le \e)} \right) \1_{\overline{\xi}_{i,n}(j) \le \e }\right\vert \ge \eta \text{ and } \forall t\in (0,1], \ X(t)\in (0,\infty)\right) =0.
\]
\end{lemma}
\begin{proof}
    The fact that $(M^{n,\e,\delta}_k)_{0\le k \le n}$ is a martingale stems from the definition of $\alpha_{i,n,\e}$. Let $0\le s<r \le 1$. By Doob's martingale inequality, 
    \[\E\left[\sup_{\lfloor ns\rfloor\le k \le \lfloor nt \rfloor} \left( M^{n,\e,\delta}_k-M^{n,\e,\delta}_{\lfloor ns \rfloor }\right)^2\right] \le 4\E\left[ \left( M^{n,\e,\delta}_{\lfloor nt \rfloor} -M^{n,\e,\delta}_{\lfloor ns \rfloor}\right)^2\right].\]
    Moreover, 
    \begin{align*}
         \E\left[ \left( M^{n,\e,\delta}_{\lfloor nt \rfloor} -M^{n,\e,\delta}_{\lfloor ns \rfloor}\right)^2\right]&= \E\left[\sum_{i=\lfloor ns \rfloor +1}^{\lfloor nt \rfloor} \sum_{j=1}^{Z_{i-1,n}} \left( \frac{\overline{\xi}_{i,n}(j)}{1+ \overline{\xi}_{i,n}(j)^2} - \frac{{\alpha}_{i,n,\e}}{\P(\overline{\xi}_{i,n} \le \e)} \right)^2 \1_{\overline{\xi}_{i,n}(j) \le \e\text{ and } \delta \ell_n \le Z_{i-1,n}\le \ell_n/\delta} \right]\\
         &\le \E \left[  \sum_{i=\lfloor ns \rfloor +1}^{\lfloor nt \rfloor} \sum_{j=1}^{Z_{i-1,n}} \left(\frac{\overline{\xi}_{i,n}(j)}{1+ \overline{\xi}_{i,n}(j)^2}\right)^2 \1_{\overline{\xi}_{i,n}(j) \le \e\text{ and } \delta \ell_n \le Z_{i-1,n}\le \ell_n/\delta}\right]\\
         &\le \E\left[  \sum_{i=\lfloor ns \rfloor+1}^{\lfloor nt \rfloor} \sum_{j=1}^{Z_{i-1, n}} \frac{\overline{\xi}_{i,n}(j)^2}{1+ \overline{\xi}_{i,n}(j)^2} \1_{\overline{\xi}_{i,n}(j) \le \e\text{ and } \delta \ell_n \le Z_{i-1,n}\le \ell_n/\delta}\right],
    \end{align*}
    where in the first line, we used the independence of the $\overline{\xi}_{i,n}(j)'s$ conditionally on $Z_{i-1, n}$ and in the first inequality we used the fact that the variance under $\P(\cdot \vert \overline{\xi}_{i,n}(j) \le \e)$ is smaller or equal to the second moment.

    Next, we bound from above
    \begin{align*}
        \E\left[  \sum_{i=\lfloor ns \rfloor +1}^{\lfloor nt \rfloor} \sum_{j=1}^{Z_{i-1, n}} \frac{\overline{\xi}_{i,n}(j)^2}{1+ \overline{\xi}_{i,n}(j)^2} \1_{\overline{\xi}_{i,n}(j) \le \e\text{ and } \delta \ell_n \le Z_{i-1,n}\le \ell_n/\delta}\right]
        &\le \sum_{i=\lfloor ns \rfloor+1}^{\lfloor nt \rfloor} \frac{\ell_n}{\delta} \E\left[\frac{\overline{\xi}_{i,n}^2}{1+ \overline{\xi}_{i,n}^2} \1_{\overline{\xi}_{i,n}\le \e}\right]\\
        &=\frac{\ell_n}{\delta}\sum_{i=\lfloor ns \rfloor+1}^{\lfloor nt \rfloor} \left(
        \beta_{i,n} - \E\left[\frac{\overline{\xi}_{i,n}^2}{1+ \overline{\xi}_{i,n}^2} \1_{\overline{\xi}_{i,n}> \e} \right]
        \right).
    \end{align*}
    Moreover, by \ref{hyp:A1},
    \[
    \limsup_{n\to \infty}
    \frac{\ell_n}{\delta}\sum_{i=\lfloor ns \rfloor+1}^{\lfloor nt \rfloor} \left(
        \beta_{i,n} - \E\left[\frac{\overline{\xi}_{i,n}^2}{1+ \overline{\xi}_{i,n}^2} \1_{\overline{\xi}_{i,n}> \e} \right] \right)
    \le
    \frac{2}{\delta} \left( \beta(t)-\beta(s) - \frac{1}{2}\int_{(\e, \infty) \times [s,t]} \frac{x^2}{1+x^2} \mu(\mathrm{d} x, \mathrm{d} t) \right).
    \]
    We conclude the proof of the first point of the lemma by letting $\e\to 0$.

    For the second statement of the lemma, let $\e'\in (0,1)$. Thanks to \eqref{eq cv GWVE}, we have
    \[
    \lim_{\delta \to 0} \P \left( \forall i \in [\e'n, (1-\e')n], \ \delta \ell_n \le Z_{i-1,n} \le \ell_n/ \delta \text{ and } \forall t \in (0,1], \ X(t)\in (0,\infty)\right) = \P\left( \forall t \in (0,1], \ X(t) \in (0,\infty) \right).
    \]
    So, it is enough to show that for all $\delta>0, \eta>0$,
    \[
\lim_{\e \to 0} \limsup_{n\to \infty} \P \left( \sup_{i \in [n]} \left\vert \sum_{j=1}^{Z_{i-1,n}} \left(\frac{\overline{\xi}_{i,n}(j)}{1+ \overline{\xi}_{i,n}(j)^2} - \frac{{\alpha}_{i,n,\e}}{\P(\overline{\xi}_{i,n} \le \e)} \right) \1_{\overline{\xi}_{i,n}(j) \le \e \text{ and } \delta\ell_n \le Z_{i-1, n} \le \ell_n/\delta} \right\vert\ge \eta \right) =0.
\]
Let $0=t_0<t_1<\ldots<t_r=1$ be a subdivision of $[0,1]$. Then, note that for all $\ell \in [r]$, for all $\lfloor nt_{\ell-1}\rfloor +1\le i \le nt_\ell$,
\[
\sum_{j=1}^{Z_{i-1,n}} \left(\frac{\overline{\xi}_{i,n}(j)}{1+ \overline{\xi}_{i,n}(j)^2} - \frac{{\alpha}_{i,n,\e}}{\P(\overline{\xi}_{i,n} \le \e)} \right) \1_{\overline{\xi}_{i,n}(j) \le \e \text{ and } \delta\ell_n \le Z_{i-1, n} \le \ell_n/\delta}   = \left(M^{n,\e,\delta}_{i} - M^{n,\e,\delta}_{\lfloor nt_{\ell-1} \rfloor }\right) - \left( M^{n,\e,\delta}_{i-1} - M^{n,\e,\delta}_{\lfloor nt_{\ell-1} \rfloor }\right).
\]
Hence,
\[
\left( \sum_{j=1}^{Z_{i-1,n}} \left(\frac{\overline{\xi}_{i,n}(j)}{1+ \overline{\xi}_{i,n}(j)^2} - \frac{{\alpha}_{i,n,\e}}{\P(\overline{\xi}_{i,n} \le \e)} \right) \1_{\overline{\xi}_{i,n}(j) \le \e \text{ and } \delta\ell_n \le Z_{i-1, n} \le \ell_n/\delta} \right)^2
\le 4 \sup_{\lfloor nt_{\ell-1} \rfloor \le k \le \lfloor nt_\ell\rfloor } \left( M^{n,\e,\delta}_k - M^{n,\e,\delta}_{\lfloor nt_{k-1} \rfloor} \right)^2.
\]
Thus,
\[
\sup_{i\in [n]} \left( \sum_{j=1}^{Z_{i-1,n}} \left(\frac{\overline{\xi}_{i,n}(j)}{1+ \overline{\xi}_{i,n}(j)^2} - \frac{{\alpha}_{i,n,\e}}{\P(\overline{\xi}_{i,n} \le \e)} \right) \1_{\overline{\xi}_{i,n}(j) \le \e \text{ and } \delta\ell_n \le Z_{i-1, n} \le \ell_n/\delta} \right)^2
\le \max_{\ell \in [r]} 4 \sup_{\lfloor nt_{\ell-1} \rfloor \le k \le \lfloor nt_\ell\rfloor } \left( M^{n,\e,\delta}_k - M^{n,\e,\delta}_{\lfloor nt_{k-1} \rfloor} \right)^2.
\]
One concludes using the first convergence of the lemma and the fact that $\widetilde{\beta}$ is continuous, by taking a subdivision with a step chosen small enough.
\end{proof}
}
{
Let us also state a technical lemma. It may be well known but we could not find it in the literature.
\begin{lemma}\label{lemme technique cv skorokhod}
    Let $(f_n)_{n\ge 1}$ be a sequence of càdlàg functions from $[0,1]$ to $\R$ which converges towards a càdlàg function $f:[0,1]\to \R$ for the $J_1$ topology of Skorokhod. Let $(\mathfrak{m}_n)_{n\ge 1}$ be a sequence of finite (positive) measures on $[0,1]$ converging weakly towards a measure $\mathfrak{m}$. Assume that $\mathfrak{m}$ has no atoms. Then,
    \[
    \int_0^1 f_n(x) \mathfrak{m}_n(\mathrm{d} x) \cv \int_0^1 f(x) \mathfrak{m}(\mathrm{d} x).
    \]
\end{lemma}
\begin{proof}
Since $\mathfrak{m}_n([0,1])$ converges towards $\mathfrak{m}([0,1])$, we may assume that $\mathfrak{m}_n$ and $\mathfrak{m}$ are probability measures. 
Let $Y_n$ for $n\ge 1$ and $Y$ be random variables of respective laws $\mathfrak{m}_n$ for $n\ge 1$ and $\mathfrak{m}$. Since $\mathfrak{m}_n$ converges weakly towards $\mathfrak{m}$, by Skorokhod's representation theorem, we may assume that $Y_n$ converges a.s.\@ to $Y$. Moreover, since $\mathfrak{m}$ has no atoms, the càdlàg function $f$ is a.s.\@ continuous at $Y$. As a result, by the point (b.5) of Proposition 2.1 in Chapter VI of \cite{JS03},
    \[
    f_n(Y_n) \cvps f(Y).
    \]
    So, by dominated convergence (using the fact that $f$ is bounded since it is càdlàg on $[0,1]$),
    \[
    \int_0^1 f_n(x) \mathfrak{m}_n (\mathrm{d} x) 
    = \E\left[ f_n(Y_n)\right] \cv 
    \E\left[f(Y)\right]= \int_0^1 f(x) \mathfrak{m}(\mathrm{d} x).
    \]
    This is the desired result.
\end{proof}
}

\begin{proof}[Proof of Corollary \ref{corollaire CWVE}]
    The proof boils down to checking that the conditions \eqref{eq:pas de coalescence avec proba non nulle}, \ref{hyp:naissance}, \ref{hyp:splitatomes}, \ref{hyp:atomes}, \ref{hyp:coalescence}, with \ref{hyp:tightGP} and \ref{hyp:tightGHP} for the GP convergence and the GHP convergence, are satisfied in probability. The assumption \ref{hyp:naissance} comes clearly from \eqref{eq cv GWVE}. 
    
    {\emph{Step 1: Proof of \ref{hyp:splitatomes}.}} Note that for all $A>0$, {for all $\e>0$,} if $\delta_n \to 0$,
    \begin{align*}
    \P\big(
    \exists i_1,i_2 \in [n] ,&\exists j_1, j_2 \le A \ell_n, \enskip (i_1,j_1) \neq (i_2, j_2),\enskip \vert i_1 -i_2 \vert \le n\delta_n \text{ and } \xi_{i_1,n}(j_1), \xi_{i_2,n}(j_2)\ge \e \ell_n\big)
    \\
    &\le 
    A^2 \ell_n^2\sum_{i_1=1}^n  \P(\xi_{i_1,n} \ge \e \ell_n) \sum_{i_1 - n \delta_n \le i_2 \le i_1 + n \delta_n}\P(\xi_{i_2,n} \ge \e \ell_n)\\
    &\le A^2 \mu_n([\e,\infty)\times(0,1]) \max_{1 \le i_1 \le n}  \ell_n \sum_{\substack{i_1 - n \delta_n \le i_2 \le i_1 + n \delta_n\\ {i_2 \in [n]}}}\P(\xi_{i_2,n} \ge \e \ell_n),
    \end{align*}
    and therefore, {for all $\e>0$,}
    \begin{equation}\label{eq pas deux gros degres a la meme hauteur}
        \P\left(
    \exists i_1,i_2 \in [n] ,\exists j_1, j_2 \le A \ell_n, \enskip (i_1,j_1) \neq (i_2, j_2),  \enskip \vert i_1 -i_2 \vert \le n\delta_n \text{ and } \xi_{i_1,n}(j_1), \xi_{i_2,n}(j_2)\ge \e \ell_n\right)
    \cv[n] 0,
    \end{equation}
    since $\mu((0,\infty)\times \mathrm{d}t)$ has no atoms and since $\mu_n \to \mu$ by \ref{hyp:A1}. {Moreover, on the event ``$\forall t \in (0,1], \ X(t) \in (0,\infty)$'', since $X$ is càdlàg, for all $\e>0$, there exists a random variable $Y>0$ such that $\forall t \in [\e,1-\e], \ X(t) \in (Y,1/Y)$. So, by \eqref{eq cv GWVE}, which holds almost surely, we get that with high probability, for all $ t \in [\e,1-\e]$, we have $ \frac{1}{\ell_n} Z_{\lfloor n t\rfloor, n} \in (Y,1/Y)$. Combined with \eqref{eq pas deux gros degres a la meme hauteur},} this gives directly \ref{hyp:splitatomes} with high probability. 
    
    {\emph{Step 2: Proof of \ref{hyp:atomes}.} Note that by the Poisson paradigm, thanks to the convergence of $\mu_n$ to $\mu$, we have the vague convergence in distribution of the random measure on $(0,1)\times (0,\infty) \times (0,\infty)$
    \begin{equation}\label{eq cv PPP degres}
    \sum_{i \in [n]} \sum_{j\ge 1} \delta_{(j/\ell_n, \xi_{i,n}(j)/\ell_n,i/n) }
    \cvloi
    \sum_{t\ge 0} \1_{(Y(t), \Xi(t)) \neq \partial} \delta_{(Y(t), \Xi(t),t)},
    \end{equation}
    where $(Y(t), \Xi(t),t)_{t\ge 0}$ is a Poisson point process of intensity $\mathrm{d}y \mu(\mathrm{d}x,\mathrm{d}t)$ and where $\partial$ is a cemetery point. After applying Prohorov's theorem and Skorokhod's representation theorem, we may assume that the above convergence holds jointly with \eqref{eq cv GWVE} along some subsequence which is again denoted using the index $n$.

    {
    Let $\delta>0$ which is not an atom of $\mu(\mathrm{d} x\times (0,1))$ and let $k \in \N$. Let $t_k\in (0,1)$ be the $k$-th time such that $(Y(t_k), \Xi(t_k))\neq \partial$, $Y(t_k) \le X(t_k)$ and that $\Xi(t_k)\ge \delta$. Such a time is well-defined (or does not exist) thanks to the fact that $\mu([\delta,\infty) \times (0,1))<\infty$. By \eqref{eq cv GWVE}, by \eqref{eq cv PPP degres} and since $\mu([\delta,\infty) \times (0,1))<\infty$, if we let $(I_{n,k},J_{n,k})$ be the $k$-th pair (in the lexicographic order) such that $J_{n,k} \le Z_{I_{n,k}-1,n}$ and $\xi_{I_{n,k}, n}(J_{n,k}) \ge \delta \ell_n$, then we have the almost sure convergence
    \begin{equation}\label{eq cv atomes PPP}
    (I_{n,k}/n,J_{n,k}/\ell_n, \xi_{I_{n,k}, n}(J_{n,k})/\ell_n) \cvps  (t_k, Y({t_k}), \Xi(t_k)).
    \end{equation}
    For all $i \in [n]$ and $\e>0$, let
    $
    \widetilde{\alpha}_{i,n,\e} \coloneqq {\alpha_{i,n,\e}}/{\P(\overline{\xi}_{i,n} \le \e)}$. Note that for all $\e\in (0,\delta)$,
\begin{align}
    \frac{Z_{I_{n,k},n}-Z_{I_{n,k}-1, n}}{\ell_n} =
     &\sum_{j=1}^{Z_{I_{n,k}-1, n}} \overline{\xi}_{I_{n,k}, n}(j) \1_{\overline{\xi}_{I_{n,k},n}(j) >\e}\label{premiere ligne increment Z}\\
    &+ \sum_{j=1}^{Z_{I_{n,k}-1, n}} \left(\overline{\xi}_{I_{n,k},n}(j) -\widetilde{\alpha}_{I_{n,k},n,\e} \right)\1_{\overline{\xi}_{I_{n,k},n}(j) \le \e}
    \label{deuxieme ligne increment Z}\\
    &+ \widetilde{\alpha}_{I_{n,k}, n, \e} \# \{1 \le j\le Z_{I_{n,k}-1, n}; \ \overline{\xi}_{I_{n,k},n}(j) \le \e\}.\label{troisieme ligne increment Z}
\end{align}
    }
    
    By \eqref{eq cv PPP degres}, by \eqref{eq cv atomes PPP} and since $\mu((0,\infty)\times \mathrm{d}t)$ has no atoms, almost surely, the sum on the first line \eqref{premiere ligne increment Z} has only one non-zero term when $n$ is large enough:
    \[
    \sum_{j=1}^{Z_{I_{n,k}-1, n}} \overline{\xi}_{I_{n,k}, n}(j) \1_{\overline{\xi}_{I_{n,k},n}(j) >\e} = \overline{\xi}_{I_{n,k}, n}(J_{n,k}) \cvps \Xi(t_k).
    \]
    Furthermore, by Lemma \ref{lemme martingale}, for all $\eta>0$,
    \begin{equation}\label{eq cv xi sur un plus xi carre petit saut}
    \lim_{\e\to 0} \limsup_{n\to \infty} \P\left(
    \left\vert\sum_{j=1}^{Z_{I_{n,k}-1, n}} \left(\frac{\overline{\xi}_{I_{n,k},n}(j)}{1+\overline{\xi}_{I_{n,k},n}(j)^2} -\widetilde{\alpha}_{I_{n,k},n,\e} \right)\1_{\overline{\xi}_{I_{n,k},n}(j) \le \e}\right\vert \ge \eta
    \text{ and }
    \forall t \in (0,1], \ X(t) \in (0,\infty)
    \right)
    =0.
    \end{equation}
    But note that for all $\delta>0$, on the event that $Z_{I_{n,k}-1,n} \le \ell_n / \delta$, we have
    \[
    \left\vert \sum_{j=1}^{Z_{I_{n,k}-1, n}} \frac{\overline{\xi}_{I_{n,k},n}(j)^3}{1+\overline{\xi}_{I_{n,k},n}(j)^2} \1_{\overline{\xi}_{I_{n,k},n}(j) \le \e}  \right\vert
    \le \e \sum_{j=1}^{Z_{I_{n,k}-1, n}} \frac{\overline{\xi}_{I_{n,k},n}(j)^2}{1+\overline{\xi}_{I_{n,k},n}(j)^2} \1_{\overline{\xi}_{I_{n,k},n}(j) \le \e}  
    \le \e \sum_{i=1}^n \sum_{j=1}^{\lfloor \ell_n / \delta \rfloor} \frac{\overline{\xi}_{i,n}(j)^2}{1+ \overline{\xi}_{i,n}(j)^2} ,
    \]
    and that by \ref{hyp:A1}, the expectation of the right-hand side converges towards $\e\beta(1)/\delta$, which goes to zero as $\e \to 0$. Thus, for all $\eta>0$,
    \[
    \lim_{\e\to 0}\limsup_{n\to \infty} \P \left( \left\vert \sum_{j=1}^{Z_{I_{n,k}-1, n}} \frac{\overline{\xi}_{I_{n,k},n}(j)^3}{1+\overline{\xi}_{I_{n,k},n}(j)^2} \1_{\overline{\xi}_{I_{n,k},n}(j) \le \e}  \right\vert\ge \eta \text{ and }\forall t \in (0,1], \ X(t) \in (0,\infty) \right)=0.
    \]
    By taking \eqref{eq cv xi sur un plus xi carre petit saut} into account, using that $x= x/(1+x^2)+ x^3/(1+x^2)$, we obtain that on the event that $\forall t \in (0,1], \ X(t) \in (0,\infty)$, the second sum \eqref{deuxieme ligne increment Z} converges in probability as $n\to \infty$ and then $\e\to 0$ towards $0$. In other words, for all $\eta>0$,
    \[
    \lim_{\e\to 0} \limsup_{n\to \infty} \P\left(
    \left\vert\sum_{j=1}^{Z_{I_{n,k}-1, n}} \left(\overline{\xi}_{I_{n,k},n}(j) -\widetilde{\alpha}_{I_{n,k},n,\e} \right)\1_{\overline{\xi}_{I_{n,k},n}(j) \le \e}\right\vert \ge \eta
    \text{ and }
    \forall t \in (0,1], \ X(t) \in (0,\infty)
    \right)
    =0.
    \]
    {For \eqref{troisieme ligne increment Z}, notice that on the event that $\forall t \in (0,1], \ X(t) \in (0,\infty)$, by \eqref{eq cv GWVE} we have almost surely
    \[\widetilde{\alpha}_{I_{n,k}, n, \e} \# \{1 \le j\le Z_{I_{n,k}-1, n}; \ \overline{\xi}_{I_{n,k},n}(j) \le \e\}= \widetilde{\alpha}_{I_{n,k}, n, \e} O(\ell_n).\]
    But, by \eqref{eq alpha i n epsilon tend vers zero}, we know that $ \ell_n\alpha_{I_{n,k}, n,\e} \to 0$ a.s.\@ as $n\to \infty$. Moreover, by \ref{hyp:A1},
    \[
    \sup_{i\in [n]} \left(1-\P( \overline{\xi}_{i,n} \le \e) \right) = \sup_{i\in [n]} \P(\overline{\xi}_{i,n} >\e) \le \sum_{i=1}^n \P(\overline{\xi}_{i,n} >\e) = O\left(\frac{1}{\ell_n} \right).
    \]
    Therefore, almost surely, $\ell_n \widetilde{\alpha}_{I_{n,k}, n, \e} \to 0$ as $n\to \infty$. This proves that \eqref{troisieme ligne increment Z} goes to zero almost surely as $n\to \infty$ and then $\e \to 0$ on the event that $\forall t \in (0,1], \ X(t) \in (0,\infty)$. As a result, for all $\eta>0$,
    \[
    \lim_{\e \to 0} \limsup_{n\to \infty} \P \left( \left\vert \frac{Z_{I_{n,k}, n} - Z_{I_{n,k}-1, n}}{\ell_n} - \overline{\xi}_{I_{n,k},n}(J_{n,k}) \right \vert \ge \eta \text{ and } \forall t \in (0,1], \ X(t) \in (0,\infty)\right) =0.
    \]
    Since the above probability does not depend on $\e$, we deduce that on the event that $\forall t \in (0,1], \ X(t) \in (0,\infty)$,
    \begin{equation}\label{eq cv proba grand degre}
    \left\vert \frac{Z_{I_{n,k}, n} - Z_{I_{n,k}-1, n}}{\ell_n} - \overline{\xi}_{I_{n,k},n}(J_{n,k}) \right \vert \cvproba 0.
    \end{equation}
    So, by taking \eqref{eq cv GWVE} into account, the high degrees correspond to positive jumps of $X$. Conversely, let $t'_k\in (0,1)$ be the $k$-th positive jump of $X$ of size larger than $\delta$. Let $I'_{n,k}$ be the $k$-th time $i \in [n]$ such that $(Z_{I'_{n,k},n} - Z_{I'_{n,k}-1,n})/\ell_n >\delta$. By \eqref{eq cv GWVE}, on the event that $X$ has no jump of size exactly $\delta$, we have the convergence
    \[
    \left( \frac{I'_{n,k}}{n}, \frac{Z_{I'_{n,k}, n}- Z_{I'_{n,k}-1, n}}{\ell_n}\right)
    \cvps 
    \left(t'_k, \Delta_+ X(t'_k) \right).
    \]
    Then, we can reason as above: for all $\e>0$, one can write 
    }
    
    \begin{align*}
    \frac{Z_{I'_{n,k},n}-Z_{I'_{n,k}-1, n}}{\ell_n} =
     &\sum_{j=1}^{Z_{I'_{n,k}-1, n}} \overline{\xi}_{I'_{n,k}, n}(j) \1_{\overline{\xi}_{I'_{n,k},n}(j) >\e} \\
    &+ \sum_{j=1}^{Z_{I'_{n,k}-1, n}} \left(\overline{\xi}_{I'_{n,k},n}(j) -\widetilde{\alpha}_{I'_{n,k},n,\e} \right)\1_{\overline{\xi}_{I'_{n,k},n}(j) \le \e}
    \\
    &+ \widetilde{\alpha}_{I'_{n,k}, n, \e} \# \{1 \le j\le Z_{I'_{n,k}-1, n}; \ \overline{\xi}_{I'_{n,k},n}(j) \le \e\}.
\end{align*}
    
    For the same reasons as above the second and third line converge to zero in probability as $n\to \infty$ and then $\e \to 0$ on the event that $\forall t \in (0,1], \ X(t) \in (0,\infty)$. Moreover, by \eqref{eq pas deux gros degres a la meme hauteur} we deduce that on the event that $\forall t \in (0,1], \ X(t) \in (0,\infty)$ and that $X$ has no jump of size exactly $\delta$, almost surely, for all $n$ large enough, the first sum has only one non-zero term:
    \[
    \sum_{j=1}^{Z_{I'_{n,k}-1, n}} \overline{\xi}_{I'_{n,k}, n}(j) \1_{\overline{\xi}_{I'_{n,k},n}(j) >\e}
    =
    \overline{\xi}_{I'_{n,k}, n}(J'_{n,k})
    \]
    for some $J'_{n,k} \le Z_{I'_{n,k}-1, n}$. This proves that all the jumps of $X$ correspond to high degrees, and that for all $\delta>0$ which is not an atom of $\mu(\mathrm{d} x\times (0,1))$, almost surely $X$ does not have a jump of size $\delta$. In particular, by taking \eqref{eq cv atomes PPP} and \eqref{eq cv proba grand degre} into account, we deduce that \ref{hyp:atomes} holds in probability on the event that $\forall t \in (0,1], \ X(t) \in (0,\infty)$.}
    
    {\emph{Step 3: Proof of the assumption \eqref{eq:pas de coalescence avec proba non nulle} for $\Theta$.}} In other words, we will prove that for all $0<a<b<1$, almost surely
    \begin{equation}\label{eq somme des sauts au carre}
        \sum_{t \in [a,b]} \frac{\Delta_+ X(t)^2}{X(t)^2}<\infty.
    \end{equation}
    In order to do this, let $0<a<b<1$ and let $\e,\delta>0$. On the event {that $\forall t \in [a-\delta,b+\delta], \ \delta< X_n(t) < 1/\delta$,} we have
    \begin{align*}
        \sum_{i={\lfloor an \rfloor}}^{{\lfloor bn \rfloor}} \sum_{j=1}^{Z_{i-1, n}} \left(\frac{\xi_{i,n}(j){-1}}{Z_{i,n}}\right)^2 \1_{\xi_{i,n}(j)/ Z_{i,n} > \e} 
        &\le\sum_{i=\lfloor an\rfloor}^{\lfloor bn \rfloor} \sum_{j=1}^{\ell_n/\delta} \left(\frac{\xi_{i,n}(j){-1}}{Z_{i,n}}\right)^2 \1_{\xi_{i,n}(j)/ Z_{i,n} > \e} \\
        &\le \frac{1}{\delta^2} \sum_{i=1}^n  \sum_{j=1}^{\ell_n/\delta} \left( \overline{\xi}_{i,n}(j)^2 \wedge 1\right)\1_{\overline{\xi}_{i,n}(j) > \e \delta}.
    \end{align*}
    Furthermore, using \ref{hyp:A1} which entails that $\mu_n$ converges {vaguely} towards $\mu$, we obtain that
    \begin{align*}
        {\limsup_{n\to \infty}} \ \E \left[ \sum_{i=1}^n \sum_{j=1}^{\ell_n/\delta} \left( \overline{\xi}_{i,n}(j)^2 \wedge 1\right)\1_{\overline{\xi}_{i,n}(j) > \e \delta}\right]&={\limsup_{n\to \infty}}\int_{(0,\infty) \times (0,1]}  \frac{1}{\delta} (x^2 \wedge 1) \1_{x> \e \delta} \mu_n(\mathrm{d}x,\mathrm{d}t) \\
        &{\le} \frac{1}{\delta} \int_{(0,\infty) \times(0,1]} (x^2 \wedge 1) \1_{x\ge \e \delta} \mu(\mathrm{d}x,\mathrm{d}t).
    \end{align*}
    {Moreover, by \eqref{eq cv GWVE} {and \eqref{eq cv proba grand degre}}}, when $\mu(\{\varepsilon\} \times(0,1))=0$,
    $$
    \sum_{i=\lfloor an \rfloor}^{\lfloor bn \rfloor} \sum_{j=1}^{Z_{i-1, n}} \left(\frac{\xi_{i,n}(j){-1}}{Z_{i,n}}\right)^2 \1_{\substack{\xi_{i,n}(j)/ \ell_n \ge \e \text{  and}\\ \forall t \in [a-\delta,b+\delta], \ \delta < X_n(t)< 1/\delta}} 
    \cvproba
    \sum_{t \in [a,b]} \left(\frac{\Delta_+X(t)}{X(t)} \right)^2 \1_{\substack{\Delta_+ X(t) > \e \text{  and}\\ \forall t \in [a-\delta,b+\delta], \ \delta<X(t) < 1/\delta}}.
    $$
    Thus, {by Fatou's lemma,} for all $0<a<b<1$, for all $\delta,\e>0$ such that $\mu(\{\varepsilon\} \times(0,1))=0$,
    $$\E\left[\1_{\forall t \in [a{-\delta},b{+\delta}], \ \delta <X(t) < 1/\delta }\sum_{a \le t \le b} \frac{\Delta_+ X(t)^2}{X(t)^2} \1_{\Delta_+ X(t)/X(t) \ge \e}\right] 
    \le \frac{1}{\delta^3} \int_{(0,\infty) \times (0,1)} (x^2 \wedge 1) \mu(\mathrm{d}x,\mathrm{d}t),$$
    hence \eqref{eq somme des sauts au carre} {follows from Fatou's lemma by letting $\e \to 0$} since the upper-bound does not depend on $\e$ and since $\int_{(0,\infty)\times (0,1)} x^2\wedge 1 \mu(\mathrm{d}x,\mathrm{d}t)<\infty$ by the first point of Theorem 2.1 of \cite{BS15}.  
    
    {\emph{Step 4: Proof of \ref{hyp:coalescence}.}}  For all non-negative continuous function $f:[0,1] \to \R$ with compact support $[a,b]\subset (0,1)$,
    \[
    \sum_{i =1}^{n} f\left(\frac{i}{n}\right)\sum_{j=1}^{Z_{i-1,n}} \frac{\binom{\xi_{i,n}(j)}{2}}{\binom{Z_{i,n}}{2}} 
     = \sum_{i =1}^{n} f\left(\frac{i}{n}\right) \sum_{j=1}^{Z_{i-1,n}} \frac{(\xi_{i,n}(j)-1)^2+ \xi_{i,n}(j)-1}{2\binom{Z_{i,n}}{2}} 
     = \sum_{i =1}^{n} f\left(\frac{i}{n}\right)\sum_{j=1}^{Z_{i-1,n}} \frac{\overline{\xi}_{i,n}(j)^2+ \overline{\xi}_{i,n}(j)/\ell_n}{2\binom{Z_{i,n}}{2}/\ell_n^2}.
    \]
    Let $\e,\delta>0$, {with $\delta$ small enough that $0<a-\delta<b+\delta<1$}. Let us separate the large jumps from the small jumps. More precisely, we write
    \begin{align}
        \sum_{i =1}^{n} f(i/n) \sum_{j=1}^{Z_{i-1,n}} \frac{\overline{\xi}_{i,n}(j)^2+ \overline{\xi}_{i,n}(j)/\ell_n}{2\binom{Z_{i,n}}{2}/\ell_n^2}
        &=
        \sum_{i =1}^{n} f(i/n) \sum_{j=1}^{Z_{i-1,n}} \frac{\overline{\xi}_{i,n}(j)^2+ \overline{\xi}_{i,n}(j)/\ell_n}{2\binom{Z_{i,n}}{2}/\ell_n^2} \1_{\overline{\xi}_{i,n}(j) \ge \e} \label{ligne 1 grands sauts}\\
        &+
        \sum_{i =1}^{n} f(i/n) \sum_{j=1}^{Z_{i-1,n}} \frac{\overline{\xi}_{i,n}(j)^2}{2\binom{Z_{i,n}}{2}/\ell_n^2} \1_{\overline{\xi}_{i,n}(j)<\e} \label{ligne 2 petits sauts}\\
        &+\sum_{i =1}^{n} f(i/n) \sum_{j=1}^{Z_{i-1,n}} \frac{  \overline{\xi}_{i,n}(j)/\ell_n}{2\binom{Z_{i,n}}{2}/\ell_n^2}\1_{\overline{\xi}_{i,n}(j)<\e}\label{ligne 3 petits sauts}.
    \end{align}
    We first focus on the large jumps{, i.e.\@ on the right-hand side of \eqref{ligne 1 grands sauts}}. By \eqref{eq cv GWVE} and \eqref{eq cv proba grand degre}, and on the event ``$\e$ is not the value of a positive jump of $X$'', we have
    \begin{equation}\label{eq coalescence grands sauts}
        \sum_{i =1}^{n} f(i/n) \sum_{j=1}^{Z_{i-1,n}} \frac{\overline{\xi}_{i,n}(j)^2+  \overline{\xi}_{i,n}(j)/\ell_n}{2\binom{Z_{i,n}}{2}/\ell_n^2} \1_{\overline{\xi}_{i,n}(j) \ge \e}
        \cvproba[n]
        \sum_{t \in (0,1)} f(t) \frac{\Delta_+ X(t)^2}{X(t)^2} \1_{\Delta_+ X(t)\ge \e}.
    \end{equation}
    We then focus on the small jumps. We first focus on {\eqref{ligne 2 petits sauts}}. Let us write
    $$
    G_{n,\e}\coloneqq \sum_{i =1}^{n} f(i/n) \sum_{j=1}^{Z_{i-1,n}} \frac{\overline{\xi}_{i,n}(j)^2/(1+\overline{\xi}_{i,n}(j)^2)}{2\binom{Z_{i{-1},n}}{2}/\ell_n^2} \1_{\overline{\xi}_{i,n}(j)<\e} .
    $$
    Then,
    $$
    G_{n,\e} \le \sum_{i =1}^{n} f(i/n)\sum_{j=1}^{Z_{i-1,n}} \frac{\overline{\xi}_{i,n}(j)^2}{2\binom{Z_{i{-1},n}}{2}/\ell_n^2} \1_{\overline{\xi}_{i,n}(j) < \e}
    \le (1+\e^2)G_{n,\e}.
    $$
    Moreover, if for all $i \in[n]$, we set $\widetilde{\beta}_{i,n,\e} = \E[\overline{\xi}_{i,n}^2/(1+\overline{\xi}_{i,n}^2) \1_{\overline{\xi}_{i,n} <\e}]$, then 
    \begin{align*}
    \E&\left[\left(G_{n,\e} - \sum_{i =1}^{n} f(i/n) Z_{i-1,n} \frac{\widetilde{\beta}_{i,n,\e}}{2\binom{Z_{i{-1},n}}{2}/\ell_n^2}  \right)^2 \1_{\forall t\in [a-\delta,b+\delta], \ \delta \le X_n(t) \le 1/\delta}\right]\\
    &{
    =\E\left[ \left(\sum_{i=1}^n f(i/n) \sum_{j=1}^{Z_{i-1, n}} \frac{(\overline{\xi}_{i,n}(j)^2/(1+\overline{\xi}_{i,n}(j)^2))\1_{\overline{\xi}_{i,n}(j)<\e} - \widetilde{\beta}_{i,n,\e}}{2\binom{Z_{i{-1},n}}{2}/\ell_n^2} \right)^2 \1_{\forall t\in [a-\delta,b+\delta], \ \delta \le X_n(t) \le 1/\delta} \right] }\\
    &{\le
    \E\left[ \left(\sum_{i=1}^n f(i/n) \sum_{j=1}^{Z_{i-1, n}} \frac{(\overline{\xi}_{i,n}(j)^2/(1+\overline{\xi}_{i,n}(j)^2))\1_{\overline{\xi}_{i,n}(j)<\e} - \widetilde{\beta}_{i,n,\e}}{2\binom{Z_{i{-1},n}}{2}/\ell_n^2} \1_{\ell_n \delta \le Z_{i-1,n}\le \ell_n/\delta}\right)^2  \right] }\\
    &{
    =\E\left[
    \sum_{i=1}^n f(i/n)^2 \sum_{j=1}^{Z_{i-1, n}} \frac{((\overline{\xi}_{i,n}(j)^2/(1+\overline{\xi}_{i,n}(j)^2))\1_{\overline{\xi}_{i,n}(j)<\e} - \widetilde{\beta}_{i,n,\e})^2}{\left(2\binom{Z_{i{-1},n}}{2}/\ell_n^2 \right)^2}
    \1_{\ell_n \delta \le Z_{i-1,n}\le \ell_n/\delta} \right]
    }\\
    &{\le
    \E\left[
    \sum_{i=1}^n f(i/n)^2 \sum_{j=1}^{Z_{i-1, n}} \frac{((\overline{\xi}_{i,n}(j)^2/(1+\overline{\xi}_{i,n}(j)^2))\1_{\overline{\xi}_{i,n}(j)<\e})^2}{\left(2\binom{Z_{i{-1},n}}{2}/\ell_n^2 \right)^2}
    \1_{\ell_n \delta \le Z_{i-1,n}\le \ell_n/\delta} \right]
    },
    \end{align*}
    {where in the third line, we used the fact that $f$ has compact support $[a,b]${ and $X_{n}(t)=Z_{\lfloor nt\rfloor ,n}/\ell_n$},  in the fourth line we developed the square and used the independence of the $\overline{\xi}_{i,n}(j)$'s conditionally on $Z_{i-1,n}$ and in the fifth line we used the fact that the variance is smaller or equal to the second moment. Next, the above expression is smaller or equal to }
    \begin{align*}
    {\max(f)^2
    \E\left[
    \sum_{i=1}^n  \sum_{j=1}^{Z_{i-1, n}} \frac{\e^2(\overline{\xi}_{i,n}(j)^2/(1+\overline{\xi}_{i,n}(j)^2))\1_{\overline{\xi}_{i,n}(j)<\e}}{\delta^4}
    \1_{\ell_n \delta \le Z_{i-1,n}\le \ell_n/\delta} \right]
    }
    &\le
    \max(f)^{2} \frac{\ell_n}{\delta} \sum_{i=1}^n \frac{\beta_{i,n} \e^2}{\delta^{4}}\\
    &\cv[n]
    2\max(f)^{2} \frac{\e^2}{\delta^{5}} \beta(1),
    \end{align*}
    where the convergence comes from \ref{hyp:A1}.

    Furthermore, if for all $t \in [0,1]$, we set 
    $$\widetilde{\beta}_{n,\e}(t)  = \frac{1}{2}\ell_n \sum_{i=1}^{\lfloor nt \rfloor} \widetilde{\beta}_{i,n,\e}\qquad \text{and} \qquad \widetilde{\beta}_\e(t) = \beta(t) - \frac{1}{2}\int_{(\e,\infty)\times(0,t)} \frac{x^2}{1+x^2} \mu(\mathrm{d}x,\mathrm{d}t),$$
    then by \ref{hyp:A1}, we know that $\widetilde{\beta}_{n,\e}(\mathrm{d}t)$ converges {vaguely on $((0,1),\mathcal{B}((0,1)))$} towards $\widetilde{\beta}_\e(\mathrm{d}t)$ {for all $\e>0$ which is not an atom of $\mu( \mathrm{d} x \times (0,1))$}. As a consequence, by \eqref{eq cv GWVE}, since $\widetilde{\beta}_\e$ converges {vaguely on $((0,1),\mathcal{B}((0,1)))$} towards $\widetilde{\beta}$ as $\e \to 0$ {and since $\widetilde{\beta}(\mathrm{d} t)$ has no atoms},
    $$
    \sum_{i=1}^n f(i/n) Z_{i-1,n} \frac{\widetilde{\beta}_{i,n,\e}}{2\binom{Z_{i{-1},n}}{2}/\ell_n^2}
    = \int_0^1 f\left( t\right) \frac{Z_{\lfloor nt \rfloor-1,n}}{\ell_n} \frac{2\widetilde{\beta}_{n,\e}(\mathrm{d}t)}{2\binom{Z_{\lfloor nt \rfloor{-1},n}}{2}/\ell_n^2}
    {\mathop{\longrightarrow}\limits_{n\to \infty, \e \to 0}^{\mathrm{a.s.}}\int_0^1 f(t) \frac{2\widetilde{\beta} (\mathrm{d}t)}{X(t)},}
    $$
    {in the sense that the limit as $\e \to 0$ of the limsup as $n\to \infty$ of the difference is a.s.\@ zero.}
    Thus, we deduce that, on the event ``$\forall t \in (0,1], \ X(t) \in (0, \infty)$'',
    \begin{equation}\label{eq cv petits sauts1}
        \sum_{i=1}^n f(i/n) \sum_{j=1}^{Z_{i-1,n}} \frac{\overline{\xi}_{i,n}(j)^2}{2\binom{Z_{i{-1},n}}{2}/\ell_n^2} \1_{\overline{\xi}_{i,n}(j)<\e}
        \mathop{\longrightarrow}\limits_{n\to \infty, \e \to 0}^\P\int_0^1 f(t) \frac{2\widetilde{\beta} (\mathrm{d}t)}{X(t)},
    \end{equation}
{in the sense that for all $\eta>0$, the limsup as $n\to \infty$ of the probability that the difference is larger in absolute value than $\eta$ goes to zero as $\e \to 0$.}
    {
    But we want the above convergence with a $Z_{i,n}$ instead of $Z_{i-1,n}$. To that aim, let us write
    \begin{equation}\label{eq def rho n epsilon}
    \widetilde{\rho}_{n,\e}(\mathrm{d}t)\coloneqq \sum_{i=1}^n \delta_{i/n}(\mathrm{d} t) \sum_{j=1}^{Z_{i-1,n}} \frac{\overline{\xi}_{i,n}(j)^2}{2\binom{Z_{i{-1},n}}{2}/\ell_n^2} \1_{\overline{\xi}_{i,n}(j)<\e}.
    \end{equation}
    Then \eqref{eq cv petits sauts1} implies the vague convergence in probability on $(0,1)$ of $\widetilde{\rho}_{n,\e}(\mathrm{d}t)$ towards $2 \widetilde{\beta}(\mathrm{d} t)/ X(t)$ as $n\to \infty$ and then $\e \to 0$. 
    Let $
    g_n(t) \coloneqq 
    f(t) {\binom{Z_{\lfloor nt\rfloor-1, n}}{2}}/{\binom{Z_{\lfloor n t\rfloor, n}}{2}}
    $. On the event ``$\forall t \in (0,1], \ X(t) \in (0, \infty)$'', the sequence of random functions $g_n$ with compact support $[a,b]$ converges a.s.\@ for the $J_1$ topology of Skorokhod towards $g(t) \coloneqq f(t) {X(t-)^2}/{X(t)^2}$. So, by Lemma \ref{lemme technique cv skorokhod},
    \begin{equation}\label{eq cv integrale gn rho n epsilon}
        \int_0^1 g_n(t) \widetilde{\rho}_{n,\e}(\mathrm{d}t) \mathop{\longrightarrow}\limits_{n\to \infty, \e \to 0}^{\mathrm{\P}} \int_0^1 g(t) \frac{2 \widetilde{\beta}(\mathrm{d}t)}{X(t)} = \int_0^1 f(t) \frac{X(t-)^2}{X(t)^2}\frac{2 \widetilde{\beta}(\mathrm{d}t)}{X(t)} = \int_0^1 f(t) \frac{2 \widetilde{\beta}(\mathrm{d}t)}{X(t)},
    \end{equation}
    where in the last equality, we use the facts that $\widetilde{\beta}$ has no atoms and the fact that for every $t$ except for a random countable subset we have $X(t-)=X(t)$.
    In other words,
    \begin{equation}\label{eq cv petits sauts}
        \sum_{i=1}^n f(i/n) \sum_{j=1}^{Z_{i-1,n}} \frac{\overline{\xi}_{i,n}(j)^2}{2\binom{Z_{i,n}}{2}/\ell_n^2} \1_{\overline{\xi}_{i,n}(j)<\e}
        \mathop{\longrightarrow}\limits_{n\to \infty, \e \to 0}^\P\int_0^1 f(t) \frac{2\widetilde{\beta} (\mathrm{d}t)}{X(t)}.
    \end{equation}}
    Let us then show that {\eqref{ligne 3 petits sauts} is negligible, i.e.\@ that}
    \begin{equation}\label{eq reste tend vers zero}
    \sum_{i =1}^{n} f(i/n) \sum_{j=1}^{Z_{i-1,n}} \frac{  \overline{\xi}_{i,n}(j)/\ell_n}{2\binom{Z_{i,n}}{2}/\ell_n^2}\1_{\overline{\xi}_{i,n}(j)<\e} 
    \mathop{\longrightarrow}\limits_{n\to \infty, \e \to 0}^{\P} 0.
    \end{equation}
    In order to show \eqref{eq reste tend vers zero}, let us first show that 
    \begin{equation}\label{eq reste tend vers zero2}
    \sum_{i =1}^{n} f(i/n) \sum_{j=1}^{Z_{i-1,n}} \frac{  \overline{\xi}_{i,n}(j)/\ell_n}{2\binom{Z_{i-1,n}}{2}/\ell_n^2}\1_{\overline{\xi}_{i,n}(j)<\e} 
    \cvproba[n] 0.
    \end{equation}
    Using that {$(1+1/\ell_n^2) \overline{\xi}_{i,n}(j) \le (1+\overline{\xi}_{i,n}(j)^2) \overline{\xi}_{i,n}(j)$ and that} for all $x\le \e$, we have $x\le x/(1+x^2)+ \e x^2/(1+x^2)$, we {bound}
    \begin{align}
    &\frac{1+1/\ell_n^2}{\ell_n}\sum_{i =1}^{n} f(i/n) \sum_{j=1}^{Z_{i-1,n}} \frac{  \overline{\xi}_{i,n}(j)/(1+\overline{\xi}_{i,n}(j)^2)}{2\binom{Z_{i{-1},n}}{2}/\ell_n^2}\1_{\overline{\xi}_{i,n}(j)<\e}\le{
    \sum_{i =1}^{n} f(i/n) \sum_{j=1}^{Z_{i-1,n}} \frac{  \overline{\xi}_{i,n}(j)/\ell_n}{2\binom{Z_{i-1,n}}{2}/\ell_n^2}\1_{\overline{\xi}_{i,n}(j)<\e} } \notag \\
    &\le \frac{1}{\ell_n}\sum_{i =1}^{n} f(i/n) \sum_{j=1}^{Z_{i-1,n}} \frac{  \overline{\xi}_{i,n}(j)/(1+\overline{\xi}_{i,n}(j)^2)}{2\binom{Z_{i{-1},n}}{2}/\ell_n^2}\1_{\overline{\xi}_{i,n}(j)<\e} + \frac{\e}{\ell_n}\sum_{i =1}^{n} f(i/n) \sum_{j=1}^{Z_{i-1,n}} \frac{  \overline{\xi}_{i,n}(j)^2/(1+\overline{\xi}_{i,n}(j)^2)}{2\binom{Z_{i{-1},n}}{2}/\ell_n^2}\1_{\overline{\xi}_{i,n}(j)<\e}.
    \label{encadrement somme petits sauts}
    \end{align}
    By \eqref{eq cv petits sauts1} and since $\ell_n \to \infty$, we know that the above term on the right goes to zero as $n\to \infty$ in probability.

    Moreover,
\begin{align*}
\Bigg\vert\E\left[\frac{1}{\ell_n}\sum_{i =1}^{n} { \1_{\delta \ell_n \le Z_{i-1,n}\le \ell_n/\delta} }f(i/n) \sum_{j=1}^{Z_{i-1,n}} \frac{  \overline{\xi}_{i,n}(j)/(1+\overline{\xi}_{i,n}(j)^2)}{2\binom{Z_{i{-1},n}}{2}/\ell_n^2} \right] \Bigg\vert &=\left\vert \E\left[\frac{1}{\ell_n}\sum_{i =1}^{n} { \1_{\delta \ell_n \le Z_{i-1,n}\le \ell_n/\delta} }f(i/n) {Z_{i-1,n}} \frac{  \alpha_{i,n} }{2\binom{Z_{i{-1},n}}{2}/\ell_n^2}\right] \right\vert\\
& \le \E\left[\frac{1}{\ell_n}\sum_{i =1}^{n} { \1_{\delta \ell_n \le Z_{i-1,n}\le \ell_n/\delta} }f(i/n) {Z_{i-1,n}} \frac{  \vert \alpha_{i,n} \vert }{2\binom{Z_{i{-1},n}}{2}/\ell_n^2}\right] \\
&\le\frac{1}{\ell_n} \max(f) \frac{1}{\delta} \frac{\lVert \alpha_n \rVert(1)}{\delta^2} \cv[n] 0,
\end{align*}
{where the last inequality comes from the fact that $f$ has compact support $[a,b]$} where the convergence stems from \ref{hyp:A1} and from the fact that $\ell_n \to \infty$. 

{Furthermore, we bound from above the variance as follows:
\begin{align*}
    \E&\left[\left( \frac{1}{\ell_n}\sum_{i=1}^n f(i/n) \sum_{j=1}^{Z_{i-1,n}} \frac{\overline{\xi}_{i,n}(j)/(1+ \overline{\xi}_{i,n}(j)^2) - \alpha_{i,n}}{2 \binom{Z_{i-1,n}}{2}/ \ell_n^2} \1_{\delta \ell_n \le Z_{i-1,n} \le \ell_n/ \delta}\right)^2 \right]\\
    &\le \frac{1}{\ell_n^2} \E\left[\sum_{i=1}^n f(i/n)^2 \sum_{j=1}^{Z_{i-1,n}} \frac{\left(\overline{\xi}_{i,n}(j)/(1+ \overline{\xi}_{i,n}(j)^2) - \alpha_{i,n}\right)^2}{4 \binom{Z_{i-1,n}}{2}^2/ \ell_n^4} \1_{\delta \ell_n \le Z_{i-1,n} \le \ell_n/ \delta} \right]\\
    &\le \frac{1}{\ell_n^2} \max(f)^2\E\left[ \sum_{i=1}^n \sum_{j=1}^{Z_{i-1,n}} \frac{\overline{\xi}_{i,n}(j)^2/(1+ \overline{\xi}_{i,n}(j)^2)^2}{4 \binom{Z_{i-1,n}}{2}^2/ \ell_n^4} \1_{\delta \ell_n \le Z_{i-1,n} \le \ell_n/ \delta} \right]\\
    &\le \frac{1}{\ell_n^2} \max(f)^2\E\left[ \sum_{i=1}^n \sum_{j=1}^{Z_{i-1,n}} \frac{\overline{\xi}_{i,n}(j)^2/(1+ \overline{\xi}_{i,n}(j)^2)}{4 \binom{Z_{i-1,n}}{2}^2/ \ell_n^4} \1_{\delta \ell_n \le Z_{i-1,n} \le \ell_n/ \delta} \right]\\
    &\le \frac{1}{\ell_n^2} \max(f)^2 \E\left[ \sum_{i=1}^n \frac{\ell_n}{\delta} \frac{\beta_{i,n}}{\delta^4}\right]\\
    &=O\left( \frac{1}{\ell_n^2} \right).
\end{align*}
Therefore, by letting $\delta\to 0$, on the event ``$\forall t \in (0,1], \ X(t) \in (0,\infty)$'',
\[
\frac{1}{\ell_n}\sum_{i =1}^{n}  f(i/n) \sum_{j=1}^{Z_{i-1,n}} \frac{  \overline{\xi}_{i,n}(j)/(1+\overline{\xi}_{i,n}(j)^2)}{2\binom{Z_{i{-1},n}}{2}/\ell_n^2}
\cvproba 0.
\]
Besides, by \eqref{eq cv GWVE} and \eqref{eq cv proba grand degre}, we know that for all $\e>0$ which is not the size of a positive jump of $X$:
\[
\sum_{i =1}^{n}  f(i/n) \sum_{j=1}^{Z_{i-1,n}} \frac{  \overline{\xi}_{i,n}(j)/(1+\overline{\xi}_{i,n}(j)^2)}{2\binom{Z_{i{-1},n}}{2}/\ell_n^2} \1_{\overline{\xi}_{i,n}(j)\ge\e}
\cvproba \sum_{t \in (0,1)} f(t) \frac{\Delta_+X(t)/(1+\Delta_+X(t)^2)}{X(t)^2} \1_{\Delta_+ X(t) \ge \e},
\]
so that
\[
\frac{1}{\ell_n}\sum_{i =1}^{n}  f(i/n) \sum_{j=1}^{Z_{i-1,n}} \frac{  \overline{\xi}_{i,n}(j)/(1+\overline{\xi}_{i,n}(j)^2)}{2\binom{Z_{i{-1},n}}{2}/\ell_n^2} \1_{\overline{\xi}_{i,n}(j)\ge\e}
\cvproba 0.
\]
Thus, 
\[
\frac{1}{\ell_n}\sum_{i =1}^{n}  f(i/n) \sum_{j=1}^{Z_{i-1,n}} \frac{  \overline{\xi}_{i,n}(j)/(1+\overline{\xi}_{i,n}(j)^2)}{2\binom{Z_{i{-1},n}}{2}/\ell_n^2} \1_{\overline{\xi}_{i,n}(j) < \e}
\cvproba 0.
\]
Thanks to \eqref{encadrement somme petits sauts}, we deduce \eqref{eq reste tend vers zero2}.

Next, we need to check that \eqref{eq reste tend vers zero2} implies \eqref{eq reste tend vers zero}. Note that since $\vert \overline{\xi}_{i,n}(j)/\ell_n \vert \le \overline{\xi}_{i,n}(j)^2$, we have
\begin{align*}
    &\left\vert\sum_{i=1}^n f(i/n) \sum_{j=1}^{Z_{i-1,n}} \frac{\overline{\xi}_{i,n}(j)}{\ell_n} \1_{\overline{\xi}_{i,n}(j)<\e} \left( \frac{1}{2 \binom{Z_{i-1,n}}{2}/\ell_n^2}- \frac{1}{2 \binom{Z_{i,n}}{2}/\ell_n^2}\right) \right\vert\\
    &\le \sum_{i=1}^n f(i/n) \sum_{j=1}^{Z_{i-1,n}} \overline{\xi}_{i,n}(j)^2 \1_{\overline{\xi}_{i,n}(j)<\e} \left\vert \frac{1}{2 \binom{Z_{i-1,n}}{2}/\ell_n^2}- \frac{1}{2 \binom{Z_{i,n}}{2}/\ell_n^2}\right\vert= \int_0^1 r_n(t) \widetilde{\rho}_{n,\e}(\mathrm{d}t),
\end{align*}
where we recall that $\widetilde{\rho}_{n,\e}$ is defined in \eqref{eq def rho n epsilon} and we set for all $t \in [0,1]$,
\[r_n(t)\coloneqq f(t) \left\vert 1- \frac{\binom{Z_{\lfloor nt \rfloor -1, n}}{2}} {\binom{Z_{\lfloor nt \rfloor , n}}{2}}\right\vert.\]
Then, by the same reasoning as in \eqref{eq cv integrale gn rho n epsilon}, we obtain the convergence
\[
\int_0^1 r_n(t) \widetilde{\rho}_{n,\e} (\mathrm{d}t) \mathop{\longrightarrow}\limits_{n\to \infty, \e \to 0}^\P 0.
\]
Hence,
\[\left\vert\sum_{i=1}^n f(i/n) \sum_{j=1}^{Z_{i-1,n}} \frac{\overline{\xi}_{i,n}(j)}{\ell_n} \1_{\overline{\xi}_{i,n}(j)<\e} \left( \frac{1}{2 \binom{Z_{i-1,n}}{2}/\ell_n^2}- \frac{1}{2 \binom{Z_{i,n}}{2}/\ell_n^2}\right) \right\vert
\mathop{\longrightarrow}\limits_{n\to \infty, \e \to 0}^\P 0.
\]
Thus, we have shown \eqref{eq reste tend vers zero}.
}

Combining \eqref{eq coalescence grands sauts}, \eqref{eq cv petits sauts}, \eqref{eq reste tend vers zero}, we conclude that on the event ``$\forall t \in (0,1], \ X(t) \in (0,\infty)$'',
    $$
    \sum_{i =1}^{n} f(i/n)\sum_{j=1}^{Z_{i-1,n}} \frac{\binom{\xi_{i,n}(j)}{2}}{\binom{Z_{i,n}}{2}} \cvproba[n]\sum_{t\in (0,1)} f(t) \frac{\Delta_+ X(t)^2}{X(t)^2} + \int_0^1 f(t) \frac{2 \widetilde{\beta}(\mathrm{d}t)}{X(t)}.
    $$
    This proves \ref{hyp:coalescence}.
    
    {\emph{Step 5: Checking \ref{hyp:tightGP} and \ref{hyp:tightGHP}.}} The assumption that for all $0<a<b<1$, we have $\int_{(0,\infty)\times[a,b]} x \mu(\mathrm{d}x,\mathrm{d}t) = \infty$ or ${\widetilde{\beta}}(b)-{\widetilde{\beta}}(a)>0$ entails readily \ref{hyp:tightGP}. Indeed, when $\int_{(0,\infty)\times[a,b]} x \mu(\mathrm{d}x,\mathrm{d}t) = \infty$, almost surely, for all $\delta>0$, we have $
    \sum_{a\le t \le b} \Xi(t) \1_{Y(t) \le \delta}=\infty
    $, where $(\Xi(t), Y(t))_{t\ge 0}$ is the Poisson point process introduced in \eqref{eq cv PPP degres}, so that on the event ``$\forall t \in (0,1], \ X(t) \in (0,\infty)$'', a.s.
    \[
    \sum_{a \le t \le b} \Xi(t) \1_{Y(t) \le X(t)} = \infty.
    \]
    Thus, on the event ``$\forall t \in (0,1], \ X(t) \in (0,\infty)$'', a.s.
    \[
    \sum_{a \le t \le b} \frac{\Delta_+ X(t)}{X(t)} = \infty.
    \] 
    Finally, under the assumption \eqref{eq hypothese GWVE tight GHP}, for all $\delta\in (0,1)$, the condition \ref{hyp:tightGHP} is satisfied with probability at least $1-\delta$. 
\end{proof}

\bibliographystyle{alpha}
\bibliography{bibli}

\appendix

\section{Background on the GP, GH and GHP topologies}
\label{sec:background}
\subsection{The Gromov--Prokhorov (GP) topology} \label{GPdef}
A measured metric space is a triple $(X,d,\mu)$ such that $(X,d)$ is a Polish space and $\mu$ is a Borel probability measure on $X$. Two such spaces $(X,d,\mu)$, $(X',d',\mu')$ are called GP-isometry-equivalent if and only if there exists an isometry $f:\supp(\mu)\to \supp(\mu')$ such that if $f_\star \mu$ is the image of $\mu$ by $f$ then $f_\star \mu=\mu'$. 
Let $\mathbb{K}_{\GP}$ be the set of GP-equivalent classes of measured metric spaces. Given a measured metric space $(X,d,\mu)$, we write $[X,d,\mu]$ for the GP-isometry-equivalence class of $(X,d,\mu)$ and frequently use the notation $X$ for either $(X,d,\mu)$ or $[X,d,\mu]$.

We now recall the definition of the Prokhorov distance. Consider a metric space $(X,d)$. For every $A\subset X$ and $\e>0$ let $A^\e\coloneqq \{x\in X, d(x,A)<\e\}$ be the open $\varepsilon$-neighborhood of $A$. Then given two (Borel) probability measures $\mu$, $\nu$ on $X$, the Prokhorov distance between $\mu$ and $\nu$ is defined by 
\[d_P(\mu, \nu)\coloneqq \inf\{\text{ $\e>0$: $\mu(A)\leq \nu (A^\e)+\e$ and $\nu(A)\leq  \mu(A^\e)+\e$, for all Borel set $A\subset X$} \}.\]

The  Gromov--Prokhorov (for short GP) distance is an extension of the Prokhorov's distance: For every $[X,d,\mu],[X',d',\mu']\in \K_{\GP}$ the Gromov--Prokhorov distance between $X$ and $X'$ is defined by
\[ d_{\GP}([X,d,\mu],[X',d',\mu'])\coloneqq\inf_{S,\phi,\phi'} d_P(\phi_\star \mu, \phi'_\star\mu'),\]
where the infimum is taken over all metric spaces $S$ and isometric embeddings $\phi :X\to S$, $\phi' :X'\to S$. $d_{\GP}$ is indeed a distance on $\K_{\GP}$ and $(\K_{\GP},d_{\GP})$ is a Polish space (see e.g. \cite{ADH13}).

We use another convenient characterization of the GP topology using the convergence of distance matrices: For every measured metric space  $(X,d^X,\mu^X)$ let $(x_i^X)_{i\in \N}$ be a sequence of i.i.d. random variables of common distribution $\mu^X$ and let $M^X\coloneqq(d^X(x_i^X,x_j^X))_{i,j\in \N}$. We have the following result from \cite{Loh13},

\begin{lemma} \label{equivGP} Let $(X^n)_{n\in \N} \in \K_{\GP}^\N$ and let $X\in \K_{\GP}$ then $X^n\limit^{\GP}X$ as $n\to \infty$ if and only if $M^{X^n}$ converges in distribution toward $M^X$ for the product topology.
\end{lemma}
\subsection{The Gromov--Hausdorff (GH) topology} \label{GH}
Let $\K_{\GH}$ be the set of isometry-equivalent classes of compact metric spaces. For every metric space $(X,d)$, we write $[X,d]$ for the isometry-equivalence class of $(X,d)$, and frequently write $X$ for either $(X,d)$ or $[X,d]$. 

For every metric space $(X,d)$, the Hausdorff distance between $A,B\subset X$ is given by
\[d_H(A,B)\coloneqq \inf\{\e>0, A\subset B^\e, B\subset A^\e \}. \]
The Gromov--Hausdorff distance between $[X,d]$,$[X',d']\in \K_{\GH}$ is given by 
\[ d_{\GH}([X,d],[X',d'])\coloneqq\inf_{S,\phi,\phi'} \left (d_H(\phi(X), \phi'(X')) \right ),\]
where the infimum is taken over all metric spaces $S$ and isometric embeddings $\phi :X\to S$, $\phi' :X'\to S$. $d_{\GH}$ is indeed a distance on $\K_{\GH}$ and $(\K_{\GH},d_{\GH})$ is a Polish space (see e.g.  \cite{ADH13}).

\subsection{The Gromov--Hausdorff--Prokhorov (GHP) topology}
Two measured metric spaces $(X,d,\mu)$, $(X',d',\mu')$ are called GHP-isometry-equivalent if and only if there exists an isometry $f:X\to X'$ such that if $f_\star \mu$ is the image of $\mu$ by $f$ then $f_\star \mu=\mu'$.
Let $\K_{\GHP}\subset\K_{\GP}$ be the set of isometry-equivalence classes of compact measured metric spaces.

The Gromov--Hausdorff--Prokhorov distance between $[X,d,\mu]$,$[X',d',\mu']\in \K_{\GHP}$ is given by 
\[ d_{\GHP}([X,d,\mu],[X',d',\mu'])\coloneqq\inf_{S,\phi,\phi'} \left ( d_P(\phi_\star \mu, \phi'_\star\mu')+d_H(\phi(X), \phi'(X')) \right ),\]
where the infimum is taken over all metric spaces $S$ and isometric embeddings $\phi :X\to S$, $\phi' :X'\to S$. $d_{\GHP}$ is indeed a distance on $\K_{\GHP}$ and $(\K_{\GHP},d_{\GHP})$ is a Polish space (see  \cite{ADH13}).
\section{Leaf-tightness criteria}
\label{sec:leaf}
In this section, $\mathbf X=((X^n,d^n,p^n))_{n\in \N}$ denotes  a sequence of random compact measured metric spaces GHP-measurable. For every $n\in \N$, let $(x_i^n)_{i\in \N}$ be a sequence of i.i.d.\@ random variables of common distribution $p^n$, then let $M^n\coloneqq(d^n(x_i^n,x_j^n))_{i,j\in \N}$. 

We say that $\mathbf X$ is weak-leaf-tight if and only if
\begin{equation} \forall \delta>0, \lim_{k\to \infty} \limsup_{n\to \infty} \proba{d^n(x^{n}_{k+1},\{x^n_1,x^n_2,\ldots,x^n_k\})>\delta} = 0. \label{eq:weak-leaf-tight} \end{equation}
We say that $\mathbf X$ is strong-leaf-tight if and only if
\begin{equation} \forall \delta>0, \lim_{k\to \infty} \limsup_{n\to \infty} \proba{d_H(X^n,\{x^n_1,x^n_2,\ldots,x^n_k\})>\delta } = 0. \label{eq:strong-leaf-tight} \end{equation}
Those criteria were first introduced by Aldous \cite{Ald91,Ald93}.

\begin{proposition}
\label{B.1}
If $(M^n)_{n\in \N}$ converges weakly toward a random matrix $M$, $\mathbf X$ is weak-leaf-tight, and for every $n\in \N$, $(X^n,d^n)$ is a random tree equipped with random edge lengths, then $\mathbf X$ converges weakly for the GP topology toward a random measured $\R$-tree $(X,d,p)$. Furthermore, if $(x_i)_{i\in \N}$ are i.i.d. random variables of law $p$ then $M^X\coloneqq(d(x_i,x_j))_{i,j\in \N}=^{(d)} M$.
 \end{proposition}
 \begin{proof} The result is directly adapted from Theorem 3 of Aldous \cite{Ald93} using modern formalism. 
 \end{proof}
\begin{proposition}[Proposition B.1 from \cite{BBKKbis23+}]
\label{B.2}
If $(M^n)_{n\in \N}$ converges weakly toward a random matrix $M$, and $\mathbf X$ is strong-leaf-tight, then $\mathbf X$ converges weakly for the GHP topology toward a random compact measured metric space $(X,d,p)$. Furthermore, if $(x_i)_{i\in \N}$ are i.i.d. random variables of law $p$ then $M^X\coloneqq(d(x_i,x_j))_{i,j\in \N}=^{(d)} M$. In addition, a.s. $p$  has full support.
 \end{proposition}
\end{document}